\pdfoutput=1
\documentclass[11pt,leqno]{article}

\usepackage{amsmath,amsthm,amssymb}
\usepackage{graphicx,subfigure,booktabs,enumerate,url}
\usepackage{bm}
\usepackage{enumitem}
\usepackage{float}

\usepackage{geometry}
\theoremstyle{plain}
\newtheorem{thm}{Theorem}
\newtheorem{prop}[thm]{Proposition}
\newtheorem{cor}[thm]{Corollary}

\theoremstyle{remark}
\newtheorem{rem}[thm]{Remark}

\theoremstyle{definition}
\newtheorem{ex}[thm]{Example}

\newcounter{algorithm}
\renewcommand{\thealgorithm}{\arabic{algorithm}}

\numberwithin{figure}{section}
\numberwithin{table}{section}

\title{\Large\bf
An inverse-stochastic-control Galerkin estimator \\
for drift potentials}
\author{Yumiharu Nakano\\[1em]
        \small{Department of Mathematical and Computing Science, School of Computing} \\
        \small{Institute of Science Tokyo} \\
        \small{W8-28, 2-12-1, Ookayama, Meguro-ku, Tokyo 152-8550, Japan} \\
                \small{e-mail: nakano@comp.isct.ac.jp}
}

\date{\today}

\begin{document}

\maketitle

\begin{abstract}
\noindent
We consider the ill-posed inverse problem of recovering the drift
potential of an overdamped Langevin diffusion from discrete, possibly
irregular position observations.  Specializing the suboptimality gap
for inverse stochastic control to drift-potential recovery, we
parametrize the associated value function in a finite
basis.  Under a basis nondegeneracy condition, the resulting Galerkin
estimator is obtained from one positive-definite linear system whose
entries are smooth functionals of the empirical time marginals; no
transition density is required.  The gap equals one half
of the squared energy distance to the true potential, yielding a
best-approximation identity and energy-seminorm consistency as the basis
becomes complete.  At a fixed basis, we separate the sample fluctuation
from the time-quadrature error, thereby avoiding the
Euler--Maruyama $O(\Delta t)$ drift-discretization bias.  We also analyze
the finite basis as regularization, give a dimension-dependent
convergence rate for a concrete basis family, and bound observation-noise
bias.  One- and two-dimensional experiments verify the predicted
sampling-interval behavior.  An application to irregular animal
telemetry illustrates the method on a single long trajectory.

\begin{flushleft}
\end{flushleft}
\end{abstract}

\section{Introduction}\label{sec:1}

We consider the inverse problem of identifying an unknown drift
potential~$U$ of the overdamped Langevin diffusion
\begin{equation}
\label{eq:1.1}
 dX_t = \nabla U(t, X_t)\, dt + \sigma\, dW_t,
 \qquad t\in[0,T],
\end{equation}
from a discrete-time, possibly irregularly spaced sample of one or more
trajectories.
Here $U:[0,T]\times\mathbb{R}^d\to\mathbb{R}$ may depend on time,
$\sigma>0$ is a scalar diffusion coefficient, and $\{W_t\}$ is a
standard $d$-dimensional Brownian motion.
This is a continuous-time inverse problem for stochastic differential
equations.
It arises across several modeling contexts, including molecular
dynamics and stochastic particle systems
\cite{frishman-ronceray:2020,bruckner-etal:2020},
parameter and drift estimation for multiscale and nonparametric
diffusion models \cite{pavliotis-stuart:2007,papaspiliopoulos-etal:2012},
data-driven Bayesian SDE inference
\cite{pokern-etal:2009,ruttor-etal:2013}, and movement ecology
\cite{kays-etal:2015,hoo-etal:2017}.
The methodology developed below is application-neutral.
We use a publicly available animal-telemetry dataset as the concrete
illustrative application in \S\ref{sec:9}.

Classical approaches include time-discretized likelihood or drift
regression \cite{bri:2010,pre-etal:2013,gloaguen-etal:2018},
occupancy-density and nonparametric drift estimators
\cite{bri-etal:2001,papaspiliopoulos-etal:2012,pokern-etal:2009,
gobet-etal:2004,comte-etal:2007,denis-etal:2021}, and weak-form
identification of the Fokker--Planck equation
\cite{lu-etal:2024,messenger-bortz:2022}.  Time-discretized likelihoods
incur an $O(\Delta t)$ drift bias \cite{florens-zmirou:1989,kessler:1997}
and become problematic for coarse, irregular records, while exact
transition densities are rarely available.  Occupancy-based methods
avoid this transition approximation but generally identify only the
gradient component of the drift.

We formulate the recovery of $U$ as an inverse stochastic control
problem using the suboptimality gap of \cite{nak:2026}.  Under a
quadratic kinetic cost, the potential is the negative value function,
$U=-v$.  We expand $v_\alpha=\sum_{j=1}^J\alpha_j\psi_j$ and minimize the
gap over the induced cost class.  The criterion is a strictly convex
quadratic form, so the inverse problem reduces to the positive-definite
linear system $B\alpha^*=-A$ and
\begin{equation}
\label{eq:1.2}
 \hat U(t,x) = -\sum_{j=1}^J \alpha^{*}_j\,\psi_j(t,x),
\end{equation}
whose entries use only empirical time marginals.  Irregular times enter
through quadrature weights rather than a transition density.

Our main contribution is a statistical and numerical analysis of this
finite-dimensional specialization.  At a fixed basis, the sample
fluctuation of order $\sqrt{\log N/N}$ separates additively from the
$O((\Delta t)^p)$ time-quadrature error on possibly irregular
observation nodes; the Euler--Maruyama $O(\Delta t)$
drift-discretization bias is absent.\footnote{The
$O(\cdot)$ notation is deterministic for quadrature terms and in
probability for sample fluctuations.}  Here $p$ is not a free
parameter: it is the order supported jointly by the quadrature rule and
the temporal regularity of the moment integrands.  The trapezoidal rule
gives $p=2$ on arbitrary observation nodes; larger $p$ requires
$C^p$ temporal regularity and a node pattern supporting a higher-order
nonnegative rule, as specified in Theorem~\ref{thm:5.3}.  We make the
amplification by the Galerkin conditioning explicit, combine basis
approximation and estimation errors, derive a growing-basis convergence
rate for a concrete family, and bound the perturbation caused by
position noise.  The action-gap identity gives the accompanying
Galerkin best-approximation characterization and population
consistency.  In the stationary limit the system is the Galerkin weak
form of the Fokker--Planck equation and coincides with closed-form score
matching \cite{hyvarinen:2007}.

At the population level, our criterion is the entropic action-matching
loss of \cite{nek-etal:2023} under $s=-v_\alpha$; we do not claim a new
variational loss.  Action matching learns the canonical curl-free drift
compatible with a prescribed marginal curve, typically through a neural
parametrization.  We instead embed the loss in inverse stochastic
control, restrict the value function to a finite Galerkin space, and
obtain a closed-form linear solve with stability and statistical error
bounds for discrete, possibly irregular observation times.  In the
stationary case a related linear variational estimator appears in
\cite{batz-etal:2016}.
One- and two-dimensional benchmarks test these claims, including a
time-dependent example and a stationarity-preserving rotational drift.
A public African-buffalo GPS record illustrates the estimator on one
long, irregularly sampled trajectory
\cite{getz-etal:2007,cross-etal:2016}.

Section~\ref{sec:methods} develops the framework and the algorithm,
Section~\ref{sec:theory} gives the error analysis and the stationary
case, Section~\ref{sec:9} reports the synthetic benchmarks and the
telemetry illustration, and Section~\ref{sec:10} discusses scope,
related estimators and extensions.

\section{Methods}\label{sec:methods}

\subsection{Inverse problem formulation}\label{sec:2}

We adopt the suboptimality-gap formulation of inverse stochastic
control of \cite{nak:2026}, which we summarize here so that the
present paper is self-contained.

\paragraph{Forward control problem.}
Let $(\Omega,\mathcal{F},\mathbb{P})$ be a complete probability space carrying
an $\mathbb{R}^d$-valued random variable $X_0$ and the Brownian motion
$\{W_t\}_{t\in[0,T]}$ of \eqref{eq:1.1}, and let
$\{\mathcal{F}_t\}_{t\in[0,T]}$ denote the augmented natural
filtration generated by $X_0$ and $W$.
Denote by $\mathcal{U}$ the class of $\mathbb{R}^d$-valued,
$\mathcal{F}_t$-progressively measurable controls $\{u_t\}$ with
$\mathbb{E}\int_0^T|u_t|^2\,dt < \infty$.
We call $u\in\mathcal{U}$ an \emph{admissible} control.
We replace the observation model \eqref{eq:1.1} by the controlled
diffusion driven by 
\begin{equation}
\label{eq:2.2}
 dX_t^u = u_t\, dt + \sigma\, dW_t,\qquad X_0^u=X_0
\end{equation}
for $u\in\mathcal{U}$.
The observed trajectory of \eqref{eq:1.1} then corresponds, in this
controlled framework, to the closed-loop feedback choice
$u_t = \nabla U(t,X_t^{u})$ in $\mathcal{U}$.
A \emph{cost pair} $(f,g)$ consists of a running cost
$f:[0,T]\times\mathbb{R}^d\times\mathbb{R}^d\to\mathbb{R}$ and a terminal
cost $g:\mathbb{R}^d\to\mathbb{R}$ such that 
the associated forward stochastic control problem has optimal value
$\mathcal{J}^*(f,g) = \inf_{u\in\mathcal{U}}\mathcal{J}(u;f,g)$, where
$\mathcal{J}(u;f,g)=\mathbb{E}\bigl[\int_0^T f(t,X_t^u,u_t)\, dt + g(X_T^u)\bigr]$. 
We adopt the following two standing hypotheses.
\begin{itemize}[leftmargin=2em]
\item[(H1)] $X_0$ is independent of $W$ and $\mathbb{E}|X_0|^2<\infty$.
The SDE \eqref{eq:1.1} has a pathwise unique strong solution $X$ on
$[0,T]$, and $\mathbb{E}\int_0^T|\nabla U(t,X_t)|^2\,dt<\infty$.
\item[(H2)] $f(t,x,u)$ is continuous and has quadratic
growth in $(x,u)$, $g$ is continuous and has quadratic growth in $x$,
and the forward optimal value is finite, $\mathcal{J}^\ast(f,g)>-\infty$.
\end{itemize}
Under (H1) the observed control $u^\ast_t=\nabla U(t,X_t)$ is
progressively measurable and square integrable, so $u^\ast\in\mathcal{U}$
and $X=X^{u^\ast}$.
It also gives $\mathbb{E}\sup_{t\le T}|X_t|^2<\infty$, so the law of
$X_t$ has finite second moment uniformly in $t$.
No growth condition on $U$ is needed.
The finiteness in (H2) is what makes the gap \eqref{eq:2.5} a
well-defined criterion, and the cost class used in \S\ref{sec:4}
satisfies it by \eqref{eq:4.3}.

\paragraph{Inverse problem via the suboptimality-gap functional.}
We view the observed control $u^\ast_t=\nabla U(t,X_t)$ as the optimal
control of \eqref{eq:2.2} for an unknown cost pair $(f^\ast,g^\ast)$.
Section~\ref{sec:3} constructs such a pair and shows that $u^\ast$ is
optimal for it under the integrability conditions stated there.
To identify $(f^\ast,g^\ast)$, and hence $U$, \cite{nak:2026}
introduces the suboptimality-gap functional
\begin{equation}
\label{eq:2.5}
 V(f,g) = \int_0^T\langle f(t,\cdot,\cdot),\mu_t\rangle\, dt
        + \langle g,\tilde\mu_T\rangle - \mathcal{J}^*(f,g),
\end{equation}
where $\tilde\mu_t=\operatorname{Law}(X_t)$ and
$\mu_t=\operatorname{Law}(X_t,u_t^\ast)$ are the state and state-control
marginal laws of the observed trajectories, respectively, and
$\langle\,\cdot\,,\,\mu\rangle$ denotes the duality pairing
$\langle\varphi,\mu\rangle:=\int\varphi(z)\,\mu(dz)$ between a Borel
measure $\mu$ and a $\mu$-integrable function $\varphi$ on the
underlying space.
$V(f,g)$ measures the gap between the empirical cost of the observed
trajectories under $(f,g)$ and the optimal cost under $(f,g)$.
The inverse problem of \cite{nak:2026} is
\begin{equation}
\label{eq:2.6}
 V^* = \inf_{(f,g)\in\mathcal{D}} V(f,g),
\end{equation}
over a candidate cost class $\mathcal{D}$.
The following nonnegativity property underlies the criterion \cite{nak:2026}. 
\begin{prop}[Suboptimality-gap characterization]
\label{prop:2.1}
Under (H1) and (H2), $V(f,g)\ge 0$ for all $(f,g)\in\mathcal{D}$,
with equality if and only if the observed control is optimal for the
forward problem under $(f,g)$, that is, it attains the optimal value
$\mathcal{J}^{*}(f,g)$.
\end{prop}

\begin{proof}
By (H1), $u^\ast\in\mathcal{U}$ and $X=X^{u^\ast}$, so the first two terms
of \eqref{eq:2.5} equal $\mathcal{J}(u^\ast;f,g)$, which is finite by (H1) and (H2).
Since $\mathcal{J}^*(f,g)\le \mathcal{J}(u^\ast;f,g)$, we get
$V(f,g)=\mathcal{J}(u^\ast;f,g)-\mathcal{J}^*(f,g)\ge0$, with equality exactly when
$u^\ast$ attains $\mathcal{J}^*(f,g)$.
\end{proof}

Proposition~\ref{prop:2.1} justifies \eqref{eq:2.6} as an
identification criterion.
The minimizer is the cost pair in
$\mathcal{D}$ that best explains the observed dynamics in the
suboptimality metric, and $V^*=0$ characterizes exact agreement.
This is the inverse-control formulation we use.
It generalizes the path-error minimization of the classical inverse
optimal control literature
\cite{ng-russell:2000,ziebart-etal:2008,ab-azar-etal:2020}
to a marginal-based, continuous-time SDE setting.
Earlier work by the present author \cite{nak:2023} considered the well-posedness of this identification program.

\subsection{Quadratic control cost and the canonical correspondence}\label{sec:3}

We restrict attention to candidate costs of the form
\begin{equation}
\label{eq:3.1}
 f(t,x,u) = \Phi(t,x) + \tfrac{1}{2}|u|^2,\quad g(x)=G(x),
\end{equation}
with $\Phi\in C([0,T]\times\mathbb{R}^d)$ and $G\in C(\mathbb{R}^d)$
satisfying the quadratic growth condition of (H2).
The quadratic control penalty $\tfrac{1}{2}|u|^2$ is the standard convex
cost in stochastic control \cite{yon-zho:1999}, and in animal movement it
reflects the observation that locomotor energetic expense increases with
ground speed \cite{hal:2016}.
The associated HJB equation is
\begin{equation}
\label{eq:3.2}
 -\partial_t v - \tfrac{\sigma^2}{2}\Delta v + \tfrac{1}{2}|\nabla v|^2
   = \Phi,\quad v(T,x)=G(x),
\end{equation}
and the optimal Markov feedback is $-\nabla v(t,x)$
\cite{yon-zho:1999}.

If now $U\in C^{1,2}([0,T]\times\mathbb{R}^d)$, then setting
\begin{align}
\label{eq:3.3}
 \Phi^\ast(t,x) &= \partial_t U(t,x) + \tfrac{\sigma^2}{2}\Delta U(t,x) + \tfrac{1}{2}|\nabla U(t,x)|^2, \\
\label{eq:3.4}
 G^\ast(x) &= -U(T,x),
\end{align}
the function $v^\ast(t,x):=-U(t,x)$ solves \eqref{eq:3.2} with
$\Phi=\Phi^\ast$ and $G=G^\ast$ by direct substitution, and its feedback
$-\nabla v^\ast(t,x) = \nabla U(t,x)$ coincides with the drift of the
observation model \eqref{eq:1.1}.
Let $f^\ast=\Phi^\ast+\tfrac12|u|^2$ and $g^\ast=G^\ast$.
For $u\in\mathcal{U}$, It\^o's formula applied to $v^\ast(t,X^u_t)$ and
\eqref{eq:3.2} give
\begin{equation}
\label{eq:canon-sq}
 \mathcal{J}(u;f^\ast,g^\ast)=\mathcal{J}(u^\ast;f^\ast,g^\ast)
 +\tfrac12\,\mathbb{E}\int_0^T|u_t-\nabla U(t,X^u_t)|^2\,dt
\end{equation}
whenever the stochastic integral is a martingale and the costs are
integrable.
This holds for every $u\in\mathcal{U}$ when, for example, $U$,
$\partial_tU$ and $\Delta U$ have at most quadratic growth and $\nabla U$
has at most linear growth, and then $u^\ast$ is optimal.
This is the \emph{canonical correspondence} between the potential $U$
and the value function $v$ that underlies the inverse-control
formulation below.

\begin{rem}
\label{rem:3.2}
The canonical correspondence \eqref{eq:3.3}--\eqref{eq:3.4} shows that,
under the quadratic kinetic cost, the canonical pair is available when
$\mathcal{D}$ is large enough.
If $(\Phi^\ast, G^\ast)\in\mathcal{D}$ and \eqref{eq:canon-sq} holds for
every $u\in\mathcal{U}$, then $V^*=0$ and this pair is a minimizer.
This does not by itself imply uniqueness, since other cost pairs may
also have zero gap.
The interesting case is when $\mathcal{D}$ is smaller and does not
contain $(\Phi^\ast, G^\ast)$.
Minimizing $V$ over $\mathcal{D}$ then returns the cost pair in
$\mathcal{D}$ that makes the observed dynamics as close to optimal as
possible in the suboptimality metric.  For the value-side class used
below, (P2) makes the coefficient unique, and
Proposition~\ref{prop:ls} identifies the minimizer with the best
energy-seminorm approximation of $-U$.
\end{rem}

\subsection{Cost-side and value-side parametrizations of $\mathcal{D}$}
\label{sec:4-prelim}

We specify $\mathcal{D}$ through a finite-dimensional parametrization.
Two natural choices are available.

The first is the \emph{cost-side} option.
We expand the running cost
$\Phi_\beta(t,x)=\sum_{j=1}^m\beta_j\,\phi_j(t,x)$ on a basis
$\{\phi_j\}$ and minimize $V(f_\beta,g_\beta)$ over $\beta$.
This is the standard inverse-control choice
\cite{ng-russell:2000,ziebart-etal:2008,ab-azar-etal:2020,nak:2023},
and under the Hopf--Cole transformation it recovers the classical
Schr\"odinger problem \cite{leo:2014,mik:2021a}, sketched in
Remark~\ref{rem:cost-side}.
We do not develop it further here.

The second is the \emph{value-side} option.
We expand the value function $v(t,x)=\sum_{j=1}^J\alpha_j\psi_j(t,x)$
on a basis $\{\psi_j\}$ and minimize $V(f_\alpha,g_\alpha)$ over
$\alpha$.
Here $(f_\alpha,g_\alpha)$ is the cost pair whose optimal value is
$v_\alpha$.
The induced cost class $\mathcal{D}_{\Psi}$ is the image
of $\mathbb{R}^J$ under the HJB inversion
$\alpha\mapsto(\Phi_\alpha,G_\alpha)$ given by
\eqref{eq:4.1}--\eqref{eq:4.2} below.

The two parametrizations are inverse-control formulations of
\eqref{eq:2.6} on different basis families.
We analyze only the value-side option, whose estimate converges to $U$
as the basis becomes complete (Corollary~\ref{cor:complete}).
They differ in operational form.
The cost-side option produces a Schr\"odinger problem.
The value-side option produces a single positive-definite linear
system $B\alpha^*=-A$.
The latter admits a closed-form solution and a Galerkin analysis in
\S\ref{sec:5}.
We adopt the value-side option for these reasons.
The remainder of this section develops it into the linear system that
defines our estimator.

\subsection{Linear-in-basis value functions}\label{sec:4}

Figure~\ref{fig:method-schematic} previews the pipeline.

\begin{figure}[htbp]
\centering
\includegraphics[width=\textwidth]{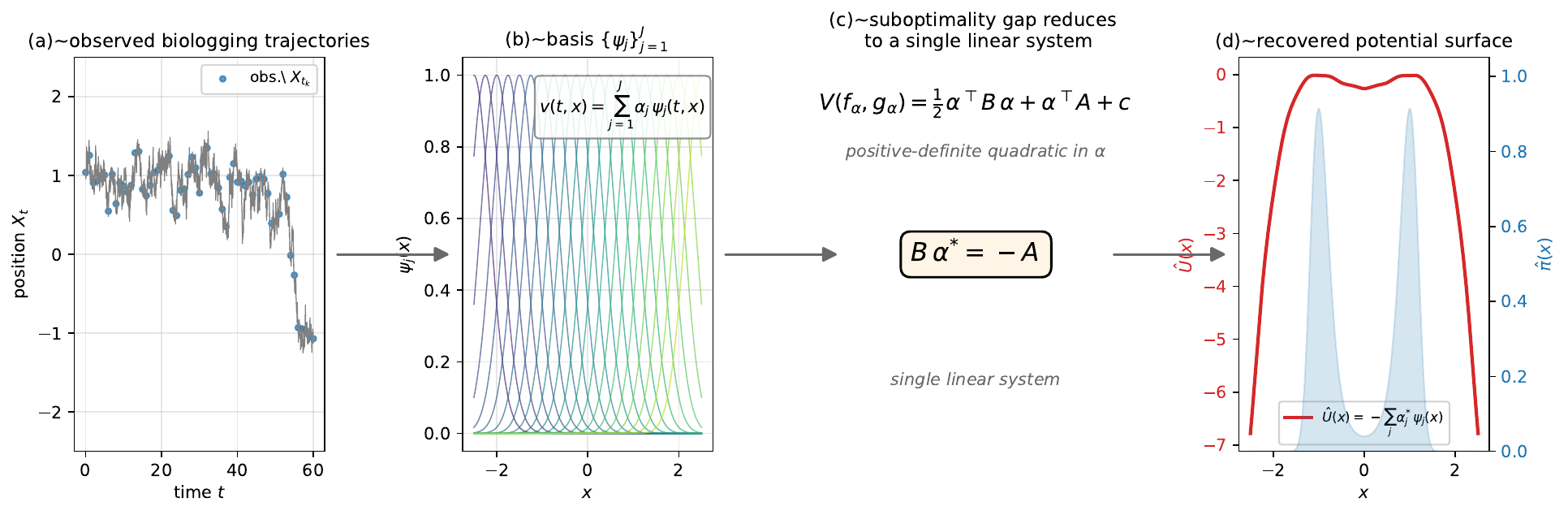}
\caption{Pipeline of the proposed estimator, illustrated on a 1D
double-well example.
(a) The input is a set of position observations $\{X_{t_k}^{(i)}\}$.
(b) A basis $\{\psi_j\}_{j=1}^J$ is chosen, and the value function is
parametrized as a linear combination $v(t,x)=\sum_j\alpha_j\psi_j(t,x)$.
(c) The suboptimality-gap functional becomes a positive-definite
quadratic form in $\alpha$, and the inverse problem reduces to a
single linear system $B\alpha^{*}=-A$.
(d) The recovered potential surface
$\hat U(x)=-\sum_j\alpha^{*}_j\psi_j(x)$
captures the high-occupancy regions of the empirical density $\hat\pi$
(here, the two wells of the double-well example) without requiring any
SDE-likelihood fit on increments.}
\label{fig:method-schematic}
\end{figure}

Let $\Psi=\{\psi_j\}_{j=1}^J\subset C^{1,2}([0,T]\times\mathbb{R}^d)$ be
a collection of basis functions satisfying the following two
conditions.
\begin{enumerate}
\item[(P1)] Each $\psi_j$, $\partial_t\psi_j$ and $\Delta\psi_j$ has at
most quadratic growth, and $\nabla\psi_j$ has at most linear growth.
\item[(P2)] The family $\{\nabla\psi_j\}_{j=1}^J$ of $\mathbb{R}^d$-valued
maps is linearly independent in $L^2$ under the measure
$\tilde\mu_t(dx)\,dt$ on $[0,T]\times\mathbb{R}^d$. Equivalently,
$\sum_{j=1}^J\alpha_j\nabla\psi_j=0$ holds $\tilde\mu_t(dx)\,dt$-almost
everywhere only when $\alpha=0$.
\end{enumerate}
For $\alpha\in\mathbb{R}^J$ we set
$v_\alpha(t,x)=\sum_{j=1}^J \alpha_j\psi_j(t,x)$.
By construction $v_\alpha$ is a candidate value function and the
recovered potential is $\hat U(t,x) = -v_\alpha(t,x)$ once
$\alpha=\alpha^{*}$ is fixed by the data-fitting step.

\paragraph{The induced cost class.}
The HJB inversion below identifies, for each $\alpha\in\mathbb{R}^J$,
a unique cost pair $(\Phi_\alpha,G_\alpha)$ whose forward optimal value
function is $v_\alpha$, so
$\mathcal{D}_{\Psi} =
\{(\Phi_\alpha+\tfrac12|u|^2,G_\alpha):\alpha\in\mathbb{R}^J\}$ is a
well-defined candidate cost class for the inverse-control problem.

For each $\alpha\in\mathbb{R}^J$, define
\begin{align}
\label{eq:4.1}
 \Phi_\alpha(t,x) &= -\sum_{j=1}^J \alpha_j\partial_t\psi_j(t,x)
   - \frac{\sigma^2}{2}\sum_{j=1}^J\alpha_j\Delta\psi_j(t,x)
   + \frac{1}{2}\Bigl|\sum_{j=1}^J \alpha_j\nabla\psi_j(t,x)\Bigr|^2, \\
\label{eq:4.2}
 G_\alpha(x) &= \sum_{j=1}^J \alpha_j\psi_j(T,x).
\end{align}
Direct substitution shows that $v_\alpha$ solves \eqref{eq:3.2} with
$\Phi=\Phi_\alpha$ and $G=G_\alpha$.
Define $u_\alpha^*(t,x)$ by
\[
 u_\alpha^*(t,x)=-\sum_j\alpha_j\nabla\psi_j(t,x).
\]
For every $u\in\mathcal{U}$, It\^o's formula applied to
$v_\alpha(t,X^u_t)$ and \eqref{eq:3.2} give
\[
 \mathcal{J}(u;f_\alpha,g_\alpha)
 =\langle v_\alpha(0,\cdot),\tilde\mu_0\rangle
 +\tfrac12\mathbb{E}\int_0^T|u_t-u^*_\alpha(t,X^u_t)|^2\,dt.
\]
Here the stochastic integral is a martingale and all costs are
integrable, by (P1) and $\mathbb{E}\sup_{t\le T}|X^u_t|^2<\infty$.
The feedback $u^*_\alpha$ is locally Lipschitz with linear growth, so it
defines an admissible control that attains the minimum, and the forward
optimal value is
\begin{equation}
\label{eq:4.3}
 \mathcal{J}^*(f_\alpha, g_\alpha) = \langle v_\alpha(0,\cdot),\tilde\mu_0\rangle
   = \sum_{j=1}^J \alpha_j \langle \psi_j(0,\cdot),\tilde\mu_0\rangle.
\end{equation}
We define the candidate cost class
\begin{equation}
\label{eq:4.4}
 \mathcal{D}_{\Psi} = \bigl\{(\Phi_\alpha+\tfrac{1}{2}|u|^2,\, G_\alpha): \alpha\in\mathbb{R}^J\bigr\}.
\end{equation}
Under (P1), each member of $\mathcal{D}_{\Psi}$ satisfies (H2), with the
finiteness given by \eqref{eq:4.3}.

The gap \eqref{eq:2.5} is evaluated on the observed dynamics, where the
control $u$ is not a separate observation but the virtually observed
control of the data, namely the drift $\nabla U$.
Its kinetic part $\tfrac12|u|^2$ then contributes only a constant to
$V(f_\alpha,g_\alpha)$, independent of $\alpha$, so the candidate
dependence enters only through $\Phi_\alpha$ and $G_\alpha$, which
integrate against the state occupancy $\tilde\mu_t$, so no observation of
the control is needed.

\begin{thm}[Quadratic form of the suboptimality gap]
\label{thm:4.3}
Under {\rm (H1)}, {\rm (P1)} and {\rm (P2)},
\begin{equation}
\label{eq:4.5}
 V(f_\alpha, g_\alpha) = \tfrac{1}{2}\alpha^{\mathsf T} B\, \alpha + \alpha^{\mathsf T} A + c
\end{equation}
where $c\in\mathbb{R}$ is independent of $\alpha$, and the vector
$A=(A_j)_{j=1}^{J}\in\mathbb{R}^{J}$ and the matrix
$B=(B_{jk})_{j,k=1}^{J}\in\mathbb{R}^{J\times J}$ have entries
\begin{align}
\label{eq:4.6}
 A_j &= \int_0^T \Bigl\langle -\partial_t \psi_j - \tfrac{\sigma^2}{2}\Delta \psi_j,\, \tilde\mu_t \Bigr\rangle\, dt
       + \langle \psi_j(T,\cdot),\tilde\mu_T\rangle - \langle \psi_j(0,\cdot),\tilde\mu_0\rangle, \\
\label{eq:4.7}
 B_{jk} &= \int_0^T \langle (\nabla\psi_j)^{\mathsf T}\nabla\psi_k,\, \tilde\mu_t\rangle\, dt.
\end{align}
In particular, $B$ is symmetric and positive definite.
\end{thm}

\begin{proof}
By \eqref{eq:2.5},
$V(f_\alpha,g_\alpha)
= \int_0^T\langle\Phi_\alpha,\tilde\mu_t\rangle\, dt
 + \tfrac{1}{2}\int_0^T\langle |u|^2,\mu_t\rangle\, dt
 + \langle G_\alpha,\tilde\mu_T\rangle
 - \mathcal{J}^*(f_\alpha,g_\alpha)$.
The middle term $c_0:=\tfrac{1}{2}\int_0^T\langle |u|^2,\mu_t\rangle\, dt$ is
independent of $\alpha$.
Substituting \eqref{eq:4.1}, \eqref{eq:4.2}, \eqref{eq:4.3} into $V(f_\alpha, g_\alpha)$, we get 
\begin{align*}
 V(f_\alpha,g_\alpha)
 &= -\sum_j\alpha_j\int_0^T\langle\partial_t\psi_j,\tilde\mu_t\rangle\, dt
  -\frac{\sigma^2}{2}\sum_j\alpha_j\int_0^T\langle\Delta\psi_j,\tilde\mu_t\rangle\, dt \\
 &\quad +\frac{1}{2}\sum_{j,k}\alpha_j\alpha_k\int_0^T\langle(\nabla\psi_j)^{\mathsf T}\nabla\psi_k,\tilde\mu_t\rangle\, dt \\
 &\quad + c_0 + \sum_j\alpha_j\langle\psi_j(T,\cdot),\tilde\mu_T\rangle - \sum_j\alpha_j\langle\psi_j(0,\cdot),\tilde\mu_0\rangle \\
 &= \tfrac{1}{2}\sum_{j,k}\alpha_j\alpha_k B_{jk} + \sum_j\alpha_j A_j + c_0,
\end{align*}
which is \eqref{eq:4.5} with $c=c_0$.
$B$ is symmetric by definition.
By (P2), for nonzero $\alpha$ we have
$\alpha^{\mathsf T} B\alpha
 = \int_0^T\langle |\sum_j\alpha_j\nabla\psi_j|^2,\tilde\mu_t\rangle\, dt
 > 0$, so $B$ is positive definite.
\end{proof}

Since $B$ is positive definite, $V(f_\alpha,g_\alpha)$ is a coercive
strictly convex quadratic in $\alpha$, and its unique minimizer over
$\mathbb{R}^J$ is
\begin{equation}
\label{eq:4.8}
 \alpha^* = -B^{-1} A,
\end{equation}
and the estimated potential is
\begin{equation}
\label{eq:4.9}
 \hat U(t,x) = -v_{\alpha^*}(t,x) = -\sum_{j=1}^J \alpha^*_j\, \psi_j(t,x).
\end{equation}

Adding a function of $t$ alone to $U$ does not change the drift, so $U$
and $\hat U$ are determined only up to such a function.

\begin{rem}[Anisotropic noise and a weighted kinetic cost]
\label{rem:aniso}
Let the controlled dynamics be $dX^u_t=u_t\,dt+\Sigma\,dW_t$ with a
constant invertible $\Sigma$, and let the kinetic cost be
$\tfrac12u^{\mathsf T}Qu$ with $Q$ symmetric positive definite.
The Hamilton--Jacobi--Bellman equation \eqref{eq:3.2} becomes
$-\partial_tv-\tfrac12\mathrm{tr}(\Sigma\Sigma^{\mathsf T}D^2v)
+\tfrac12(\nabla v)^{\mathsf T}Q^{-1}\nabla v=\Phi$ with the same
terminal condition, its optimal feedback is $-Q^{-1}\nabla v$, and
$v^\ast=-U$ then corresponds to the observation model
$dX_t=Q^{-1}\nabla U(t,X_t)\,dt+\Sigma\,dW_t$ in place of
\eqref{eq:1.1}.
Theorem~\ref{thm:4.3} holds with
\begin{align*}
 A_j &= \int_0^T\Bigl\langle-\partial_t\psi_j
   -\tfrac12\mathrm{tr}\bigl(\Sigma\Sigma^{\mathsf T}D^2\psi_j\bigr),
   \tilde\mu_t\Bigr\rangle\,dt
   + \langle\psi_j(T,\cdot),\tilde\mu_T\rangle
   - \langle\psi_j(0,\cdot),\tilde\mu_0\rangle, \\
 B_{jk} &= \int_0^T\langle(\nabla\psi_j)^{\mathsf T}Q^{-1}\nabla\psi_k,
   \tilde\mu_t\rangle\,dt,
\end{align*}
and \S\ref{sec:5} holds verbatim with the weighted energy form
$a_Q(w_1,w_2)=\int_0^T\langle(\nabla w_1)^{\mathsf T}Q^{-1}\nabla w_2,
\tilde\mu_t\rangle\,dt$ in place of $a$.
The setting of this paper is $\Sigma=\sigma I$ and $Q=I$, and the
telemetry application of \S\ref{sec:9.3} is the case $\Sigma=\sigma R$
and $Q=R^{-2}$ with a diagonal $R$.
\end{rem}

\begin{ex}[Gaussian basis]
\label{ex:4.6}
Let $\psi_j(t,x)=\eta_j(t)\,e^{-|x-c_j|^2/(2h_j^2)}$ for centers
$c_j\in\mathbb{R}^d$, bandwidths $h_j>0$, and time profiles
$\eta_j\in C^1([0,T])$, taken in $C^{p+1}([0,T])$ when the quadrature
order $p$ of \S\ref{sec:5} is used.
Both (P1) and (P2) hold whenever the spatial centers are distinct, the
time profiles sharing a common center are linearly independent, and
$\tilde\mu_t$ has a positive density on a nonempty open set for almost
every $t$.
The gradient requirement (P2) reduces to linear independence of the
$\psi_j$ themselves, since
$\sum_j\alpha_j\nabla\psi_j=\nabla\sum_j\alpha_j\psi_j$ vanishes on that
open set only when $\sum_j\alpha_j\psi_j$ is constant in $x$ there, hence
on $\mathbb{R}^d$ by analyticity, and the Gaussian decay forces that
constant to zero.
$\Delta\psi_j$ and $(\nabla\psi_j)^{\mathsf T}\nabla\psi_k$ have closed-form expressions,
and integration against $\tilde\mu_t$ reduces to expectations of
Gaussian-modulated polynomials, amenable to Monte Carlo and Gauss--Hermite
quadrature.
\end{ex}

\begin{ex}[Tensor-product time-space basis]
\label{ex:4.7}
Let $\psi_{(j,m)}(t,x)=\zeta_m(t)\,\varphi_j(x)$
with $\{\varphi_j\}$ a spatial basis (RBF, B-spline, Hermite) and
$\{\zeta_m\}$ a temporal basis (Fourier modes for periodic regimes,
polynomial, B-spline).
When $\tilde\mu_t$ does not depend on $t$, $A$ and $B$ factor across the
tensor structure, yielding sparse linear algebra, and the temporal basis
handles periodic regimes explicitly.
\end{ex}

\subsection{Algorithm}\label{sec:8}

\begin{center}
\fbox{\parbox{0.92\textwidth}{%
\refstepcounter{algorithm}\label{alg:7-1}%
\textbf{Algorithm~\thealgorithm{} (Linear-in-basis potential identification).}
\\[0.5ex]
\textbf{Input:} Trajectories $\{X_{t_k}^{(i)}\}$, $i=1,\ldots,N$, $k=0,\ldots,K$;
basis $\Psi=\{\psi_j\}_{j=1}^J$;
diffusion $\sigma$;
Tikhonov parameter $\lambda\ge 0$;
quadrature weights $\{w_k\}$.
\\[0.5ex]
\textbf{Output:} Coefficient vector $\hat\alpha^*\in\mathbb{R}^J$ and potential
estimate $\hat U(t,x)$.
\begin{enumerate}[leftmargin=2em]
\item \emph{Compute basis evaluations.} For each $j,k,i$: evaluate
$\partial_t\psi_j(t_k, X_{t_k}^{(i)})$, $\Delta\psi_j(t_k, X_{t_k}^{(i)})$,
and $\nabla\psi_j(t_k, X_{t_k}^{(i)})$.
\item \emph{Empirical marginal expectations.} For each $j,k$:
$\bar a_j^{(k)} = -\tfrac{1}{N}\sum_i\partial_t\psi_j(t_k, X_{t_k}^{(i)})
- \tfrac{\sigma^2}{2N}\sum_i\Delta\psi_j(t_k, X_{t_k}^{(i)})$;
for $k=0,K$: $\bar\psi_j^{(k)} = \tfrac{1}{N}\sum_i\psi_j(t_k, X_{t_k}^{(i)})$.
\item \emph{Empirical Gram-like terms.} For each $j,\ell,k$:
$\bar b_{j\ell}^{(k)} = \tfrac{1}{N}\sum_i\bigl[(\nabla\psi_j)^{\mathsf T}\nabla\psi_\ell\bigr](t_k,X_{t_k}^{(i)})$.
\item \emph{Time quadrature.}
$\hat A_j = \sum_k w_k \bar a_j^{(k)} + \bar\psi_j^{(K)} - \bar\psi_j^{(0)}$;
$\hat B_{j\ell} = \sum_k w_k \bar b_{j\ell}^{(k)}$.
\item \emph{Solve the linear system}
$(\hat B + \lambda I)\,\hat\alpha^* = -\hat A$,
that is, $\hat\alpha^* = -(\hat B + \lambda I)^{-1}\hat A$.
\item \emph{Return} $\hat U(t,x) = -\sum_j \hat\alpha^*_j \psi_j(t,x)$.
\end{enumerate}
}}
\end{center}

Basis evaluations cost $O(NKJd)$ and assembly of the Gram matrix costs
$O(NKJ^2d)$.
The linear system contributes an additional $O(J^3)$.
The Tikhonov parameter $\lambda$, the basis dimension $J$ and the basis
bandwidth can be selected by cross-validation across held-out
trajectories.
By Theorem~\ref{thm:4.3} and Proposition~\ref{prop:ls},
$\tfrac12\alpha^{\mathsf T}B\alpha+\alpha^{\mathsf T}A
=\tfrac12\|v_\alpha-v^\ast\|_a^2-c$ with $c$ independent of $\alpha$, so
this quantity, evaluated with $\hat A$ and $\hat B$ from held-out
trajectories, measures the energy distance to the true value function
up to a constant.
We split the trajectories into a few groups, fit $\hat\alpha^*$ on all
but one group, evaluate this criterion on the held-out group, and choose
the values minimizing its average.
Alternatively, the discrepancy principle sets the residual
$|\hat B\hat\alpha^*(\lambda)+\hat A|=\lambda|\hat\alpha^*(\lambda)|$
equal to a bound on the noise level such as $\varepsilon_A$ of
Theorem~\ref{thm:5.3}, which requires an estimate of $\kappa$.
The experiments of \S\ref{sec:9} use fixed values, $\lambda=10^{-8}$ in
the stationary benchmarks of \S\ref{sec:9.1}--\S\ref{sec:9.2},
$\lambda=10^{-2}$ in the time-dependent benchmark, and $\lambda=10^{-4}$
in the rotational experiment and for the telemetry data.

\section{Theoretical analysis and connections}\label{sec:theory}

This section collects the variational, statistical and
stochastic-control structure that underlies the estimator of
Section~\ref{sec:methods}.
Section~\ref{sec:5} casts the closed-form linear system as the Galerkin
discretization of an underlying continuous variational problem and
derives a C\'ea-type bound and a statistical error decomposition.
Section~\ref{sec:6} specializes to the stationary regime.
It identifies the estimator with the classical Fokker--Planck weak-form
recovery in the complete-basis limit.
Remark~\ref{rem:cost-side} sketches a dual cost-side parametrization
linking the framework to the Schr\"odinger problem.

\subsection{Galerkin interpretation and error analysis}\label{sec:5}

The first-order optimality condition $B\alpha^*=-A$ is the variational
equation of a Galerkin projection of a continuous problem.
We make this explicit and derive the C\'ea-type bound and the
statistical-error decomposition that exhibits $\Delta t$-decoupling.

\paragraph{The continuous variational problem.}
For functions $w_1,w_2$ on $[0,T]\times\mathbb{R}^d$ whose spatial
gradients lie in $L^2(\tilde\mu_t(dx)\,dt)$, define the bilinear form
\begin{equation}
\label{eq:5.1}
 a(w_1, w_2) = \int_0^T \langle (\nabla w_1)^{\mathsf T}\nabla w_2,\,\tilde\mu_t\rangle\, dt
\end{equation}
and, for $w$ satisfying (P1), the linear functional
\begin{equation}
\label{eq:5.2}
 L(w) = -\int_0^T\langle\partial_t w + \tfrac{\sigma^2}{2}\Delta w,\,\tilde\mu_t\rangle\, dt
                   + \langle w(T,\cdot),\tilde\mu_T\rangle - \langle w(0,\cdot),\tilde\mu_0\rangle.
\end{equation}
The form $a$ is only positive semidefinite: functions with zero spatial
gradient, including arbitrary functions of time, lie in its nullspace.
Within the above energy class, let
\[
 \mathcal{N}_a=\{z:a(z,z)=0\},\qquad [v]=v+\mathcal{N}_a.
\]
For $w$ satisfying (P1), It\^o's formula gives
\[
 L(w)=\int_0^T\langle\nabla U\cdot\nabla w,\tilde\mu_t\rangle\,dt.
\]
Hence $L(w)=0$ for every $w\in\mathcal{N}_a$ satisfying (P1).  We
therefore pose the problem on the quotient by $\mathcal{N}_a$, so that
it determines a spatial gradient rather than a particular additive
function of time.
The continuous variational problem reads
\begin{equation}
\label{eq:5.3}
 \text{find } [v^\ast]:\quad a(v^\ast, w) = -L(w)\ \text{for all } w \text{ satisfying (P1)}.
\end{equation}

\begin{prop}[Variational identity for $-U$]
\label{prop:5.1}
Under {\rm (H1)}, the class $[-U]$ satisfies \eqref{eq:5.3}.
\end{prop}

\begin{proof}
Let $w$ satisfy (P1) and apply It\^o's formula to $w(t,X_t)$.
By (P1) and (H1) the stochastic integral is a martingale and every term
is integrable, since
\[
 \mathbb{E}\int_0^T|\nabla U\cdot\nabla w|\,dt
 \le\Bigl(\mathbb{E}\int_0^T|\nabla U|^2dt\Bigr)^{1/2}
     \Bigl(\mathbb{E}\int_0^T|\nabla w|^2dt\Bigr)^{1/2}.
\]
Taking expectations gives
\[
 \langle w(T,\cdot),\tilde\mu_T\rangle-\langle w(0,\cdot),\tilde\mu_0\rangle
 =\int_0^T\Bigl\langle\partial_t w+\tfrac{\sigma^2}{2}\Delta w
 +\nabla U\cdot\nabla w,\tilde\mu_t\Bigr\rangle\,dt,
\]
which is $a(-U,w)=-L(w)$.
\end{proof}

\paragraph{Galerkin projection.}
The minimizer $\alpha^*$ given by \eqref{eq:4.8} solves
\begin{equation}
\label{eq:5.4}
 a(v_{\alpha^*}, \psi_j) = -L(\psi_j),\quad j=1,\ldots,J,
\end{equation}
which is the Galerkin discretization of \eqref{eq:5.3} on $\mathrm{span}\,\Psi$.
The classical C\'ea lemma controls the Galerkin error by a constant
multiple of the best basis approximation of the exact solution.
In our setting the bilinear form $a$ is symmetric and
positive semidefinite, so this constant equals one and the bound below
is an equality.

Let $\|\cdot\|_{a} = a(\cdot,\cdot)^{1/2}$ denote the energy seminorm on
the gradient.

\begin{prop}[Least-squares form of the gap]
\label{prop:ls}
Under {\rm (H1)}, {\rm (P1)} and {\rm (P2)}, for every $\alpha\in\mathbb{R}^J$,
\begin{equation}
\label{eq:ls-gap}
 V(f_\alpha,g_\alpha)
 = \tfrac12\int_0^T\bigl\langle|\nabla v_\alpha+\nabla U|^2,\tilde\mu_t\bigr\rangle\,dt
 = \tfrac12\|v_\alpha-v^\ast\|_a^2 .
\end{equation}
\end{prop}

\begin{proof}
Proposition~\ref{prop:5.1} with $w=\psi_j$ gives
$A_j=L(\psi_j)=\int_0^T\langle\nabla U\cdot\nabla\psi_j,\tilde\mu_t\rangle\,dt$.
The constant in Theorem~\ref{thm:4.3} is
$c=\tfrac12\int_0^T\langle|\nabla U|^2,\tilde\mu_t\rangle\,dt$, since the
observed control is $\nabla U(t,X_t)$.
Expanding the square in \eqref{eq:ls-gap} gives
$\tfrac12\alpha^{\mathsf T}B\alpha+\alpha^{\mathsf T}A+c$.
\end{proof}

The estimator is thus the least-squares projection of the drift $\nabla U$
onto $\{\nabla v_\alpha\}$ in $L^2(\tilde\mu_t(dx)\,dt)$, computed through
the weak form without knowledge of $\nabla U$.
The quadratic form of Theorem~\ref{thm:4.3} and the identity
\eqref{eq:ls-gap} are the entropic action-matching loss and action-gap
identity of \cite{nek-etal:2023} under $s=-v_\alpha$.  Thus the new point
is not the population identity itself: the inverse-control
parametrization and finite Galerkin class make the empirical objective
an exactly solvable quadratic problem, rather than a general neural
optimization problem.

\begin{cor}[C\'ea-type best-approximation identity]
\label{thm:5.2}
Under {\rm (H1)}, {\rm (P1)} and {\rm (P2)},
\begin{equation}
\label{eq:5.5}
 \|v_{\alpha^*} - v^\ast\|_{a} = \inf_{\beta\in\mathbb{R}^J}\|v_\beta - v^\ast\|_{a}.
\end{equation}
\end{cor}

\begin{proof}
By Theorem~\ref{thm:4.3}, $\alpha^*$ minimizes $V(f_\alpha,g_\alpha)$ over
$\alpha\in\mathbb{R}^J$, and by Proposition~\ref{prop:ls} this is the
minimization of $\|v_\alpha-v^\ast\|_a$.
\end{proof}

\begin{cor}[Complete-basis limit]
\label{cor:complete}
Assume {\rm (H1)}, let $\psi_1,\psi_2,\ldots$ satisfy {\rm (P1)}, and let
$\{\psi_j\}_{j=1}^J$ satisfy {\rm (P2)} for every $J$.
Let $\hat U_J$ be the estimate \eqref{eq:4.9} computed from the
population quantities $A$ and $B$ with the first $J$ basis functions.
If $\nabla U$ lies in the $L^2(\tilde\mu_t(dx)\,dt)$-closure of
$\{\nabla v:\ v\in\mathrm{span}\{\psi_j\}_{j\ge1}\}$, then
$\int_0^T\langle|\nabla\hat U_J-\nabla U|^2,\tilde\mu_t\rangle\,dt\to0$
as $J\to\infty$.
\end{cor}

\begin{proof}
By Corollary~\ref{thm:5.2} this quantity equals
$\inf_{\beta\in\mathbb{R}^J}\|v_\beta-v^\ast\|_a^2$, which is
nonincreasing in $J$ and tends to zero by the closure assumption.
\end{proof}

\paragraph{Ill-posedness and the role of $J$.}
Identifying $\nabla U$ from occupancy data amounts to differentiating
the observed marginals, and the problem is ill-posed.
For scalar diffusions observed at a fixed sampling interval, the
minimax rate of drift estimation is that of a second-order linear
inverse problem \cite{gobet-etal:2004}.
The weak form \eqref{eq:4.6}--\eqref{eq:4.7} moves the derivatives onto
the basis functions, so the data enters only through smooth integrals
and no derivative of an empirical measure is taken.
The ill-posedness reappears in the stability constant.
Recovering $\alpha^*$ from $A$ and $B$ amplifies their error by
$1/c_B$ with $c_B=\sigma_{\min}(B)$, and $c_B$ degrades as the basis is
refined, since $\{\nabla\psi_j\}$ then approaches linear dependence.
The finite basis is therefore the regularization of the inverse problem,
with $J$ in the role of the regularization parameter and the Tikhonov
shift $\lambda$ as a secondary stabilizer.
Corollary~\ref{cor:total} balances the two effects and
Proposition~\ref{prop:choice} turns the balance into a choice of $J$
with a convergence rate.
Source conditions and maximum a posteriori estimates for the same
identification problem are studied in \cite{tenbrinck-etal:2025}.

\paragraph{Statistical error decomposition.}
Let $X^{(1)},\ldots,X^{(N)}$ be independent copies of the solution of
\eqref{eq:1.1} under (H1), all observed at the common deterministic
times $0=t_0<t_1<\cdots<t_K=T$, and let
$\Delta t:=\max_k(t_{k+1}-t_k)$.
Replace the population marginals $\tilde\mu_t$ by the empirical
marginals $\hat\mu^N_{t_k}=N^{-1}\sum_{i=1}^N\delta_{X_{t_k}^{(i)}}$, and
the time integrals $\int_0^T f(t)\,dt$ in \eqref{eq:4.6}--\eqref{eq:4.7}
by the quadrature sums $\sum_{k=0}^K w_k f(t_k)$.
We assume that $w_k\ge0$, $\sum_kw_k=T$, and that the rule has order $p$
on these nodes, that is,
\begin{equation}
\label{eq:quad-order}
 \Bigl|\sum_{k=0}^K w_k\varphi(t_k)-\int_0^T\varphi(t)\,dt\Bigr|
 \le C_Q\,T\,\|\varphi^{(p)}\|_\infty\,(\Delta t)^p,
 \qquad \varphi\in C^p([0,T]).
\end{equation}
The trapezoidal rule satisfies these conditions with $p=2$ on any nodes.
Nonnegative rules of higher order need more structure of the nodes.
Let $\hat A,\hat B$ denote the resulting empirical estimates of $A,B$.
For a matrix $M$, write $\sigma_{\min}(M)$ for its smallest singular
value and $\|M\|_{\rm op}$ for its operator norm.

\begin{thm}[Finite-sample statistical error bounds]
\label{thm:5.3}
Assume {\rm (H1)}, {\rm (P1)}, {\rm (P2)}, the sampling and quadrature
conditions above, and the following conditions, where
$\zeta_j=-\partial_t\psi_j-\tfrac{\sigma^2}{2}\Delta\psi_j$ and
$Y_t\sim\tilde\mu_t$.
\begin{itemize}[leftmargin=2em]
\item[(S1)] $\kappa:=\max_j\sup_{t\in[0,T]}\max\bigl\{
\|\zeta_j(t,Y_t)\|_{\psi_1}^2,\ \|\psi_j(t,Y_t)\|_{\psi_1}^2\bigr\}<\infty$,
where $\|Z\|_{\psi_1}:=\inf\{c>0:\mathbb{E}\exp(|Z|/c)\le2\}$.
\item[(S2)] The i.i.d.\ matrices
$M_i:=\bigl[\sum_kw_k(\nabla\psi_j)^{\mathsf T}\nabla\psi_\ell
(t_k,X^{(i)}_{t_k})\bigr]_{j,\ell=1}^J$ satisfy
\[
 b_{\rm op}:=\operatorname*{ess\,sup}
 \|M_1-\mathbb{E}M_1\|_{\rm op}<\infty.
\]
\item[(S3)] The functions $t\mapsto\langle \zeta_j(t,\cdot),\tilde\mu_t\rangle$
and $t\mapsto\langle(\nabla\psi_j)^{\mathsf T}\nabla\psi_\ell(t,\cdot),
\tilde\mu_t\rangle$ lie in $C^p([0,T])$, and $D_p$ is the largest
sup-norm of their $p$-th derivatives.
\end{itemize}
Then there exists absolute constants $C$ and $N_0$ such that, for
$N\ge N_0\log(JN)$, the two inequalities
\begin{align}
\label{eq:5.6}
 |\hat A - A| &\le \varepsilon_A := C(1+T)\sqrt{J\kappa\log (JN)/N}
   + \sqrt{J}\,C_Q T D_p(\Delta t)^p, \\
 \|\hat B - B\|_{\rm op} &\le \varepsilon_B := Cb_{\rm op}\sqrt{\log(JN)/N}
   + J\,C_Q T D_p(\Delta t)^p \notag
\end{align}
hold simultaneously with probability at least $1-6/N$.
If moreover $\varepsilon_B\le c_B/2$ with $c_B:=\sigma_{\min}(B)>0$,
then on this event $\sigma_{\min}(\hat B)\ge c_B/2$, and the regularized
estimator $\hat\alpha^* = -(\hat B + \lambda I)^{-1}\hat A$ with
$\lambda\ge0$ satisfies
\begin{equation}
\label{eq:5.7}
 |\hat\alpha^* - \alpha^*|
 \le \frac{2(\varepsilon_A+\varepsilon_B|\alpha^*|)}{c_B}
 + \frac{2\lambda|\alpha^*|}{c_B+2\lambda}.
\end{equation}
The constant $D_p$ depends on $U$, $\sigma$ and the basis.
\end{thm}

\begin{rem}[On the constants]
The stability constant $c_B=\sigma_{\min}(B)$ deteriorates as the basis
gradients approach linear dependence, as for finite-element stiffness
matrices \cite{fried:1973,ern-guermond:2004}.  The Tikhonov shift keeps
$\sigma_{\min}(\hat B+\lambda I)\ge\lambda$, at the cost of the last term
in \eqref{eq:5.7}.  For bounded basis functions and derivatives, as for
the Gaussian RBFs of \S\ref{sec:9}, (S1)--(S2) hold with
$b_{\rm op}\le2TJ\max_j\sup|\nabla\psi_j|^2$; unbounded gradients require
stronger tail and matrix-moment conditions.  Condition (S3) concerns
smoothness of $U$ and the basis, not of $\tilde\mu_0$.  The generator
identity
$\frac{d}{dt}\langle\varphi,\tilde\mu_t\rangle
=\langle(\partial_t+\mathcal L_t)\varphi,\tilde\mu_t\rangle$
shows that $C^{p+1}$ time profiles suffice for the Gaussian basis when
the required derivatives of $U$ have polynomial growth.  Finally,
$\kappa,b_{\rm op},D_p,c_B$ and $|\alpha^*|$ depend on $J$; hence
\eqref{eq:5.6}--\eqref{eq:5.7} are fixed-basis bounds.  With the displayed
bound on $b_{\rm op}$,
$\varepsilon_B=O(J\sqrt{\log(JN)/N}+J(\Delta t)^p)$, so increasing $J$
requires more samples, finer time resolution, or regularization.
\end{rem}

\begin{proof}[Proof of Theorem $\ref{thm:5.3}$]
Throughout, $C$ denotes an absolute constant whose value may change
from line to line.

Step (i).
For each $j$, $\hat A_j - A_j = S^A_{N,j} + Q^A_{p,j} + B^A_{N,j}$, where
\begin{align*}
 S^A_{N,j} &:= \sum_{k=0}^K w_k\,\bigl\langle \zeta_j(t_k,\cdot),\,
   \hat\mu^N_{t_k} - \tilde\mu_{t_k}\bigr\rangle,
\\
 Q^A_{p,j} &:= \sum_{k=0}^K w_k\,\langle \zeta_j(t_k,\cdot),\tilde\mu_{t_k}\rangle
 - \int_0^T\langle \zeta_j(t,\cdot),\tilde\mu_t\rangle\, dt,
\\
 B^A_{N,j} &:= \bigl\langle\psi_j(T,\cdot),\,\hat\mu^N_T - \tilde\mu_T\bigr\rangle
 - \bigl\langle\psi_j(0,\cdot),\,\hat\mu^N_0 - \tilde\mu_0\bigr\rangle.
\end{align*}

Step (ii).
The term $S^A_{N,j}$ is the centered average of the i.i.d.\ variables
$Z_i^{(j)} := \sum_k w_k\, \zeta_j(t_k, X^{(i)}_{t_k})$.
Since $w_k\ge0$ and $\sum_kw_k=T$, the triangle inequality for the
$\psi_1$-norm and (S1) give $\|Z_i^{(j)}\|_{\psi_1}\le T\sqrt{\kappa}$,
and the centering inequality \cite[Exer.~2.7.10]{vershynin:2018} gives
$\|Z_i^{(j)}-\mathbb{E}Z_i^{(j)}\|_{\psi_1}\le CT\sqrt{\kappa}$.
Bernstein's inequality \cite[Thm.~2.8.1]{vershynin:2018} then gives
$|S^A_{N,j}|\le CT\sqrt{\kappa\log (JN)/N}$ with probability at least
$1-2/(JN)$, since for $N\ge N_0\log(JN)$ this deviation lies in the
sub-Gaussian branch of the inequality.

Step (iii).
By \eqref{eq:quad-order} and (S3), $|Q^A_{p,j}|\le C_QTD_p(\Delta t)^p$.
The term $B^A_{N,j}$ is the centered average of the i.i.d.\ variables
$\psi_j(T,X^{(i)}_T)-\psi_j(0,X^{(i)}_0)$, whose $\psi_1$-norm is at most
$2\sqrt{\kappa}$ by (S1).
As in Step (ii), $|B^A_{N,j}|\le C\sqrt{\kappa\log (JN)/N}$ with
probability at least $1-2/(JN)$.
A union bound over these two events and the $J$ entries, together with
$|\hat A-A|\le\sqrt{J}\max_j|\hat A_j-A_j|$, gives the first inequality
of \eqref{eq:5.6} with probability at least $1-4/N$.
Taking the deviation level at $\log(JN)$ rather than at $\log N$ is what
keeps the $J$ entries free of charge in the union bound.

Step (iv).
Write $\hat B-B=N^{-1}\sum_{i=1}^N(M_i-\mathbb{E}M_i)+Q^B$ with
$Q^B_{j\ell}:=\sum_k w_k\langle(\nabla\psi_j)^{\mathsf T}\nabla\psi_\ell
(t_k,\cdot),\tilde\mu_{t_k}\rangle-\int_0^T\langle(\nabla\psi_j)^{\mathsf T}
\nabla\psi_\ell(t,\cdot),\tilde\mu_t\rangle\,dt$.
One matrix $M_i$ is formed per trajectory, so no union bound over the
quadrature nodes is needed.
By (S2) the matrices $M_i-\mathbb{E}M_i$ satisfy the Bernstein condition,
and \cite[Thm.~6.17]{wainwright:2019} gives
$\|N^{-1}\sum_i(M_i-\mathbb{E}M_i)\|_{\rm op}\le Cb_{\rm op}\sqrt{\log(JN)/N}$
with probability at least $1-2/N$ for $N\ge N_0\log(JN)$.
By \eqref{eq:quad-order} and (S3),
$\|Q^B\|_{\rm op}\le\|Q^B\|_F\le J\,C_QTD_p(\Delta t)^p$.
This gives the second inequality of \eqref{eq:5.6}, and intersecting
with the event of Step (iii) gives probability at least $1-6/N$.

Step (v).
Since $\alpha^* = -B^{-1}A$,
\[
 \hat\alpha^* - \alpha^*
 = -(\hat B+\lambda I)^{-1}(\hat A - A)
 - (\hat B+\lambda I)^{-1}(\hat B - B)\alpha^*
 - \lambda(\hat B+\lambda I)^{-1}\alpha^*.
\]
The matrix $\hat B$ is symmetric and positive semidefinite because
$w_k\ge0$.
On the event above, Weyl's inequality and $\varepsilon_B\le c_B/2$ give
$\sigma_{\min}(\hat B)\ge c_B-\varepsilon_B\ge c_B/2$, so
$\|(\hat B+\lambda I)^{-1}\|_{\rm op}\le 2/(c_B+2\lambda)\le 2/c_B$.
The three terms are therefore bounded by $2\varepsilon_A/c_B$,
$2\varepsilon_B|\alpha^*|/c_B$ and $2\lambda|\alpha^*|/(c_B+2\lambda)$,
which gives \eqref{eq:5.7}.
\end{proof}

\begin{cor}[Total energy-norm error bound]
\label{cor:total}
Under the assumptions of Theorem~\ref{thm:5.3} and $\varepsilon_B\le
c_B/2$, on the event of that theorem,
\begin{equation}
\label{eq:total-error}
 \|v_{\hat\alpha^*}-v^\ast\|_a
 \le \inf_{\beta\in\mathbb{R}^J}\|v_\beta-v^\ast\|_a
 + \|B\|_{\rm op}^{1/2}
   \Bigl(\frac{2(\varepsilon_A+\varepsilon_B|\alpha^*|)}{c_B}
   + \frac{2\lambda|\alpha^*|}{c_B+2\lambda}\Bigr).
\end{equation}
\end{cor}

\begin{proof}
By the triangle inequality and Corollary~\ref{thm:5.2},
$\|v_{\hat\alpha^*}-v^\ast\|_a
\le\inf_{\beta}\|v_\beta-v^\ast\|_a+\|v_{\hat\alpha^*-\alpha^*}\|_a$.
Since $\|v_\gamma\|_a^2=\gamma^{\mathsf T}B\gamma$, the second term is at
most $\|B\|_{\rm op}^{1/2}|\hat\alpha^*-\alpha^*|$, which
Theorem~\ref{thm:5.3} bounds.
\end{proof}

\begin{rem}[Choice of the basis dimension]
\label{rem:choice-J}
Equation~\eqref{eq:total-error} balances decreasing approximation error
against an estimation term that typically grows with $J$.  Total-error
consistency requires a choice $J=J(N)$ for which the latter vanishes and
$\varepsilon_B\le c_B/2$.  Because its quadrature component is
$J C_QTD_p(\Delta t)^p$ while $c_B$ usually decreases, a growing-basis
limit generally also requires $\Delta t\to0$.  If the process starts
from its stationary distribution and both $U$ and the basis are
time-independent, the moment integrands in (S3) are constant in time.
Hence $D_p=0$, the quadrature term vanishes, and no refinement of
$\Delta t$ is required for this purpose.  Proposition~\ref{prop:choice}
makes this balance explicit for polynomially controlled basis families.
\end{rem}

\begin{prop}[Growing-basis error bound]
\label{prop:choice}
Assume the hypotheses of Theorem~\ref{thm:5.3} with $D_p=0$, take
$\lambda=0$, and let a family of bases indexed by $J$ satisfy, with
positive constants $c_1,\ldots,c_6$ independent of $J$ and exponents
$a,b,p_\kappa,p_b,m\ge0$ and $s>0$,
\begin{equation}
\label{eq:poly-family}
 c_B\ge c_1J^{-a},\quad \|B\|_{\rm op}\le c_2J^{b},\quad
 \kappa\le c_3J^{p_\kappa},\quad b_{\rm op}\le c_4J^{p_b},\quad
 |\alpha^*|\le c_5J^{m},
\end{equation}
together with $\inf_{\beta}\|v_\beta-v^\ast\|_a\le c_6J^{-s}$.
Put $q=\tfrac{b}{2}+a+\max\bigl(\tfrac{1+p_\kappa}{2},\,p_b+m\bigr)$.
Then there is a constant $C$ such that, whenever
$\varepsilon_B\le c_B/2$, on the event of Theorem~\ref{thm:5.3},
\begin{equation}
\label{eq:choice-bound}
 \|v_{\hat\alpha^*}-v^\ast\|_a
 \le c_6J^{-s}+C\,J^{q}\sqrt{\log(JN)/N}.
\end{equation}
\end{prop}

\begin{proof}
With $D_p=0$ the quadrature terms in \eqref{eq:5.6} vanish, so
$\varepsilon_A\le C(1+T)J^{(1+p_\kappa)/2}\sqrt{\log(JN)/N}$ and
$\varepsilon_B\le Cc_4J^{p_b}\sqrt{\log(JN)/N}$ by
\eqref{eq:poly-family}.
Substituting these together with
$\|B\|_{\rm op}^{1/2}\le c_2^{1/2}J^{b/2}$ and
$c_B^{-1}\le c_1^{-1}J^{a}$ into Corollary~\ref{cor:total} with
$\lambda=0$ gives \eqref{eq:choice-bound}.
\end{proof}

\paragraph{Consequence.}
The choice $J=J(N)=\lceil(N/\log N)^{1/(2(q+s))}\rceil$ balances the two
terms in \eqref{eq:choice-bound}, which are then both of order
$(N/\log N)^{-s/(2(q+s))}$.
Since $m\ge0$ we have $p_b+a\le q$, so
$\varepsilon_B/c_B\le CJ^{p_b+a}\sqrt{\log(JN)/N}\to0$, which gives
$\varepsilon_B\le c_B/2$ for $N$ large, and $N\ge N_0\log(J(N)N)$ then
holds as well, so the event of Theorem~\ref{thm:5.3} applies.
Consequently,
\[
 \|v_{\hat\alpha^*}-v^\ast\|_a
 =O\bigl((N/\log N)^{-s/(2(q+s))}\bigr)
\]
with probability at least $1-6/N$.

\begin{ex}[Periodic cubic splines on a refined mesh]
\label{ex:splines}
Consider the periodic one-dimensional analogue of
Proposition~\ref{prop:choice}.  Let $\Omega$ be a one-dimensional torus,
let $\pi$ be bounded above and away from zero, and let $S_J$ be the space of
periodic cubic splines on a uniform mesh of width
$\Delta\asymp J^{-1}$.  Use orthonormal coordinates on
$S_J/\mathbb{R}$; then (P2) holds.  Standard spline stiffness estimates give
$c_B\gtrsim J^{-1}$ and $\|B\|_{\rm op}\lesssim J$, hence $a=b=1$.
At each point only four original B-splines are nonzero, with
$\sup|\varphi_j'|=O(J)$ and $\sup|\varphi_j''|=O(J^2)$; orthogonal
projection to the reduced coordinates does not increase the corresponding
derivative-vector norms.  Hence $\kappa\lesssim J^4$ and
\[
 b_{\rm op}\le 2T\sup_x\sum_j|\varphi_j'(x)|^2\lesssim J^2,
\]
so $p_\kappa=4$ and $p_b=2$.  Galerkin stability gives
$c_B|\alpha^*|^2\le\|v_{\alpha^*}\|_a^2\le\|U\|_a^2$, hence
$m=1/2$.  For $U\in H^4(\Omega)$ the cubic-spline energy error is
$O(J^{-3})$, so $s=3$ and $q=4$.  Consequently,
\[
 J(N)\asymp(N/\log N)^{1/14},\qquad
 \|v_{\hat\alpha^*}-v^\ast\|_a
 =O\bigl((N/\log N)^{-3/14}\bigr).
\]
\end{ex}

\begin{rem}[The Gaussian family of the experiments]
\label{rem:gauss-cond}
Fixed-bandwidth Gaussians become severely ill-conditioned as their
centers grow dense, as illustrated by Table~\ref{tab:J-sweep}.
We therefore do not invoke the polynomial condition
\eqref{eq:poly-family} for this family; the experiments keep $J$ fixed
and use a Tikhonov shift.
\end{rem}

\begin{rem}[Comparison with transition and density estimators]
\label{rem:5.4}
Euler--Maruyama maximum likelihood carries an $O(\Delta t)$ drift bias
from the discretized transition density, while the time error of
Algorithm~\ref{alg:7-1} is the quadrature error of
Theorem~\ref{thm:5.3}, which vanishes in the stationary case
(\S\ref{sec:6}, \S\ref{sec:9.1}).
Like Algorithm~\ref{alg:7-1}, the log-density estimator has no
$\Delta t$ bias, and Algorithm~\ref{alg:7-1} adds a parametric basis and
the time-dependent extension.
\end{rem}

\paragraph{Observation noise.}
The analysis above treats the recorded positions as exact.
Suppose instead that the records are
$Y^{(i)}_{t_k}=X^{(i)}_{t_k}+\eta^{(i)}_k$, where the $\eta^{(i)}_k$ are
independent of $X$, independent across $i$ and $k$, centered, and have
$\mathbb{E}|\eta^{(i)}_k|^2=\tau^2$.
Let $A^Q$ and $B^Q$ denote the noise-free versions of
\eqref{eq:4.6}--\eqref{eq:4.7} in which the time integrals are replaced
by the quadrature sums.  Let $A^\eta$ and $B^\eta$ be the corresponding
quantities with $\tilde\mu_{t_k}$ replaced by the law of
$Y^{(i)}_{t_k}$, which is the convolution of $\tilde\mu_{t_k}$ with the
noise law.  Assuming $B^Q$ is positive definite, set
$c_B^Q=\sigma_{\min}(B^Q)$,
$\alpha^Q=-(B^Q)^{-1}A^Q$, and, whenever $B^\eta$ is invertible,
$\alpha^\eta=-(B^\eta)^{-1}A^\eta$.
This comparison isolates observation-noise bias from time-quadrature
error.  Notice that (P2), which concerns the continuous-time Gram
matrix, does not by itself guarantee that $B^Q$ is positive definite.

\begin{prop}[Observation-noise perturbation bounds]
\label{prop:noise}
Assume {\rm (H1)}, {\rm (P1)}, and that $B^Q$ is positive definite and
$\psi_j$, $\zeta_j$ and
$(\nabla\psi_j)^{\mathsf T}\nabla\psi_\ell$ have bounded second
derivatives in $x$ with common bound $\Lambda$.
Then
\begin{equation}
\label{eq:noise-entries}
 |A^\eta-A^Q|\le\sqrt{J}\,(T+2)\,\tfrac{\Lambda\tau^2}{2},
 \qquad
 \|B^\eta-B^Q\|_{\rm op}\le J\,T\,\tfrac{\Lambda\tau^2}{2},
\end{equation}
and if $\|B^\eta-B^Q\|_{\rm op}\le c_B^Q/2$, then $B^\eta$ is positive
definite and
\begin{equation}
\label{eq:noise-alpha}
 |\alpha^\eta-\alpha^Q|
 \le\frac{2\bigl(|A^\eta-A^Q|+\|B^\eta-B^Q\|_{\rm op}|\alpha^Q|\bigr)}{c_B^Q}.
\end{equation}
\end{prop}

\begin{proof}
For $\varphi$ with bounded second derivatives, a Taylor expansion and
$\mathbb{E}\eta^{(i)}_k=0$ give
$|\mathbb{E}\varphi(x+\eta^{(i)}_k)-\varphi(x)|
\le\tfrac12\sup|D^2\varphi|\,\tau^2$.
Applying this to $\zeta_j$, to $\psi_j(0,\cdot)$ and $\psi_j(T,\cdot)$, and
to $(\nabla\psi_j)^{\mathsf T}\nabla\psi_\ell$, and using
$\sum_kw_k=T$ for the time integrals, each entry of $A^\eta-A^Q$ is at
most $(T+2)\Lambda\tau^2/2$ and each entry of $B^\eta-B^Q$ is at most
$T\Lambda\tau^2/2$.
Then \eqref{eq:noise-entries} follows from
$|A^\eta-A^Q|\le\sqrt{J}\max_j|A^\eta_j-A^Q_j|$ and
$\|B^\eta-B^Q\|_{\rm op}\le\|B^\eta-B^Q\|_F$.
For \eqref{eq:noise-alpha}, write
$\alpha^\eta-\alpha^Q=-(B^\eta)^{-1}(A^\eta-A^Q)
 -(B^\eta)^{-1}(B^\eta-B^Q)\alpha^Q$
and repeat Step (v) of the proof of Theorem~\ref{thm:5.3} with
$\lambda=0$.
\end{proof}

The bias \eqref{eq:noise-entries} is of order $\tau^2$ and does not vanish as
$N\to\infty$: although the entries are linear functionals of the
marginal law, convolution with the noise changes the expectations of
the state-dependent nonlinear basis derivatives.
For the Gaussian basis of Example~\ref{ex:4.6} with bandwidth $h$ the
entries themselves are of order $h^{-2}$ and $\Lambda$ is of order $h^{-4}$,
so their entrywise relative perturbation is of order $(\tau/h)^2$.
The coefficient perturbation is additionally amplified by
$(c_B^Q)^{-1}$ in \eqref{eq:noise-alpha}; thus $\tau\ll h$ is reassuring
only when the discrete Gram matrix is also stable.  A noise scale
comparable with the bandwidth calls for a deconvolution or filtering
step that we do not develop here.

\subsection{Stationary case and the Fokker--Planck weak form}\label{sec:6}

When $U$ is time-independent, $U(t,x)=U(x)$ and the solution of
\eqref{eq:1.1} is non-explosive and irreducible with
$\int e^{2U/\sigma^2}dx<\infty$, the SDE admits the unique invariant
measure $\pi(dx)\propto e^{2U(x)/\sigma^2}\,dx$.
We assume the data is sampled in the stationary regime, that is, $\tilde\mu_t\equiv\pi$.

Choose $\psi_j(t,x)=\varphi_j(x)$, time-independent.
Then $\partial_t\psi_j\equiv 0$, $\psi_j(T,\cdot)-\psi_j(0,\cdot)\equiv 0$, and
\begin{align}
\label{eq:6.1}
 A_j &= -\frac{\sigma^2 T}{2} \langle\Delta\varphi_j,\pi\rangle, \\
\label{eq:6.2}
 B_{jk} &= T\langle(\nabla\varphi_j)^{\mathsf T}\nabla\varphi_k,\pi\rangle.
\end{align}
The optimality condition $B\alpha^*=-A$ becomes
\begin{equation}
\label{eq:6.3}
 \sum_k \alpha^*_k \langle(\nabla\varphi_j)^{\mathsf T}\nabla\varphi_k,\pi\rangle
 = \frac{\sigma^2}{2}\langle\Delta\varphi_j,\pi\rangle,\quad j=1,\ldots,J,
\end{equation}
with $T$ canceling out.
Because $T$ cancels, \eqref{eq:6.3} is the same for every horizon.
When the diffusion is positive recurrent with invariant law $\pi$, which
we assume whenever a single long trajectory is used, the time averages
over $[0,T]$ approach the $\pi$-averages in \eqref{eq:6.3} as
$T\to\infty$ by the ergodic theorem.
For observations at irregular deterministic or exogenous times, an
unweighted average over the recorded positions requires in addition
that the observation schedule be non-informative and that the sampled
sequence satisfy the corresponding ergodic law of large numbers.
This regime is outside Theorem~\ref{thm:5.3}, which covers independent
trajectories.

\begin{prop}[Stationary Fokker--Planck Galerkin form]
\label{prop:6.1}
For every $v\in\mathrm{span}\{\varphi_1,\ldots,\varphi_J\}$ and
$j=1,\ldots,J$, assume that, with $n_R$ the outer unit normal to
$\partial B_R$,
\[
\lim_{R\to\infty}\int_{\partial B_R}\varphi_j\pi\nabla v\cdot n_R\,dS
=\lim_{R\to\infty}\int_{\partial B_R}\varphi_j\nabla\pi\cdot n_R\,dS
=\lim_{R\to\infty}\int_{\partial B_R}\pi\nabla\varphi_j\cdot n_R\,dS=0.
\]
Then the system
\eqref{eq:6.3} is the Galerkin discretization of the elliptic equation
\begin{equation}
\label{eq:6.4}
 \nabla\cdot(\pi\,\nabla v) = -\frac{\sigma^2}{2}\Delta\pi,
\end{equation}
weak form against test functions $\varphi_j$.
\end{prop}

\begin{proof}
Multiply $-\nabla\cdot(\pi\,\nabla v)=\frac{\sigma^2}{2}\Delta\pi$ by
$\varphi_j$ and apply the integration by parts twice.
The boundary terms vanish by assumption, and then
$\int(\nabla\varphi_j)^{\mathsf T}\nabla v\,\pi\, dx = \frac{\sigma^2}{2}\int\Delta\varphi_j\,\pi\, dx$.
Substituting $v=\sum_k\alpha_k\varphi_k$, we recover \eqref{eq:6.3}.
The same identity follows from Proposition~\ref{prop:5.1} with
$\tilde\mu_t\equiv\pi$ by It\^o's formula; the boundary assumption is
needed only to identify it with the weak form of \eqref{eq:6.4}.
\end{proof}

Since $\pi\propto e^{2U/\sigma^2}$ gives $\nabla\pi=(2/\sigma^2)\pi\nabla U$,
equation \eqref{eq:6.4} is solved exactly by $v=-U$, and under the
recovery map $\hat U=-v$ this returns the true potential, $\hat U=U$.
Rewritten in $U$, equation \eqref{eq:6.4} is the classical inverse
Fokker--Planck relation
$\nabla\cdot(\pi\,\nabla U) = \frac{\sigma^2}{2}\Delta\pi$
that identifies the drift potential from a known stationary density.

In the stationary case $U=(\sigma^2/2)\log\pi$ up to an additive
constant, so under the closure assumption of
Corollary~\ref{cor:complete} the population estimate converges in the
energy seminorm to $(\sigma^2/2)\log\pi$.
This is the target of the plug-in log-density estimator
$\hat U^{\rm cl}=(\sigma^2/2)\log\hat\pi$ \cite{bri-etal:2001}, and
Theorem~\ref{thm:4.3} extends the estimator to the non-stationary case.

\begin{rem}[Relation to score matching]
\label{rem:score-matching}
Set $\theta=-(2/\sigma^2)\alpha$ and
$p_\theta\propto\exp(\sum_j\theta_j\varphi_j)$.
Since $\nabla\log\pi=(2/\sigma^2)\nabla U$, we have
\[
 \nabla\log p_\theta-\nabla\log\pi
 =-(2/\sigma^2)(\nabla v_\alpha+\nabla U),
\]
so \eqref{eq:ls-gap} becomes
\[
 V(f_\alpha,g_\alpha)
 =c+\tfrac{\sigma^4T}{4}\mathcal{J}_{\rm SM}(\theta),\qquad
 \mathcal{J}_{\rm SM}(\theta)
 =\mathbb{E}_\pi\bigl[\tfrac12|\nabla\log p_\theta|^2
 +\Delta\log p_\theta\bigr].
\]
Under the boundary condition
$\pi(x)\nabla\log p_\theta(x)\to0$, \eqref{eq:6.3} is therefore both
the closed-form score-matching estimator
\cite{hyvarinen:2005,hyvarinen:2007} and the generator moment condition
$\mathbb{E}_\pi[\mathcal{L}\varphi_j]=0$
\cite{hansen-scheinkman:1995}. With a quadratic penalty it
coincides with the stationary estimator of \cite{batz-etal:2016}.
Theorem~\ref{thm:4.3} gives the finite-horizon extension, where
$\nabla\log\tilde\mu_t$ generally differs from
$(2/\sigma^2)\nabla U(t,\cdot)$ and marginal score matching does not
recover $U$.
\end{rem}

\begin{rem}[Cost-side dual and the Schr\"odinger connection]
\label{rem:cost-side}
A complementary parametrization expands the running cost,
$\Phi_\beta=\sum_{j=1}^m\beta_j\phi_j$, and yields a
Schr\"odinger-type dual \cite{leo:2014,mik:2021a}.
The Hopf--Cole transform $\varphi=\exp(-v/\sigma^2)$ linearizes the HJB
equation, and Feynman--Kac represents $\mathcal{J}^*(f_\beta,g_\beta)$
through a path expectation that is exponential-linear in $\beta$.
Thus, under integrability, $V(f_\beta,g_\beta)$ is a convex
log-partition functional whose first-order condition matches observed
occupation statistics to those of the tilted controlled law. Unlike
the value-side system \eqref{eq:4.8}, each evaluation requires solving
a linear parabolic equation, so the minimization has no closed form.
The gap is invariant under adding a constant to $\Phi_\beta$.
\end{rem}

\section{Experiments}\label{sec:9}

\subsection{One-dimensional double-well benchmark}\label{sec:9.1}

We validate the predicted $\Delta t$-decoupling of Algorithm~\ref{alg:7-1}
on a 1D stationary potential
$U(x) = -\tfrac{1}{4}(x^2-1)^2$ with $\sigma=0.4$,
corresponding under \eqref{eq:1.1} to
$dX_t=(X_t-X_t^3)\,dt+\sigma\,dW_t$ and
$\pi(x)\propto e^{2U/\sigma^2}$, with modes at $x=\pm1$.
We simulate $N$ independent paths via Euler--Maruyama at fine step
$\Delta t_{\rm sim}=0.005$ on horizon $T=50$, then subsample at observation
intervals $\Delta t_{\rm obs}\in\{0.01, 0.1, 1.0\}$ to mimic the
coarse-$\Delta t$ regime typical of real-data SDE inference.
A Gaussian RBF basis with $J=21$ uniformly-spaced centers on $[-2.5,2.5]$ and
bandwidth $h=0.35$ serves as $\Psi$ in its stationary specialization
$\psi_j(t,x)=\varphi_j(x)$.

On the same data we compare Algorithm~\ref{alg:7-1} with the classical
log-density estimator
$\hat U^{\rm cl}=(\sigma^2/2)\log\hat\pi$ \cite{bri-etal:2001},
Euler--Maruyama drift maximum likelihood on the same basis, and its
Ozaki local-linearization correction \cite{ozaki:1992}.  The latter
freezes the drift derivative at the Euler--Maruyama fit.  For KDE we
report Silverman's bandwidth and the oracle bandwidth minimizing the
reconstruction error.
In this stationary benchmark, score matching on the same basis gives
exactly the linear system of Algorithm~\ref{alg:7-1}
(Remark~\ref{rem:score-matching}), so we do not list it separately.
The reconstruction error of each estimator $\hat U$ is reported as the
$L^2(\pi)$-norm of $\hat U - U_{\rm true}$ after subtracting the
$\pi$-weighted mean (which factors out the unrecoverable additive constant).

\begin{table}[htbp]
\centering
\caption{$L^2(\pi)$ recovery error on the 1D double-well across $N$ and
$\Delta t_{\rm obs}$. KDE uses Silverman's (Silv.) or oracle (Orc.)
bandwidth; EM-MLE and Ozaki-LL are the two transition-based fits.
Entries are means over five replications after a $20\%$ burn-in;
standard deviations are provided with the accompanying data.}
\label{tab:phase0b}
\begin{tabular}{rrrrrrr}
\toprule
$N$ & $\Delta t_{\rm obs}$ & Algorithm~\ref{alg:7-1} & KDE (Silv.) & KDE (Orc.) & EM-MLE & Ozaki-LL \\
\midrule
   100 & 0.01 & 0.0081 & 0.0097 & 0.0081 & 0.0089 & 0.0089 \\
  1000 & 0.01 & 0.0025 & 0.0037 & 0.0026 & 0.0042 & 0.0042 \\
 10000 & 0.01 & 0.0007 & 0.0014 & 0.0008 & 0.0013 & 0.0014 \\[2pt]
   100 & 0.10 & 0.0080 & 0.0135 & 0.0082 & 0.0136 & 0.0135 \\
  1000 & 0.10 & 0.0025 & 0.0061 & 0.0028 & 0.0070 & 0.0060 \\
 10000 & 0.10 & 0.0007 & 0.0027 & 0.0010 & 0.0058 & 0.0045 \\[2pt]
   100 & 1.00 & 0.0083 & 0.0210 & 0.0093 & 0.0335 & 0.0313 \\
  1000 & 1.00 & 0.0029 & 0.0113 & 0.0037 & 0.0326 & 0.0271 \\
 10000 & 1.00 & 0.0008 & 0.0054 & 0.0014 & 0.0322 & 0.0265 \\
\bottomrule
\end{tabular}
\end{table}

\begin{figure}[htbp]
\centering
\includegraphics[width=0.86\textwidth]{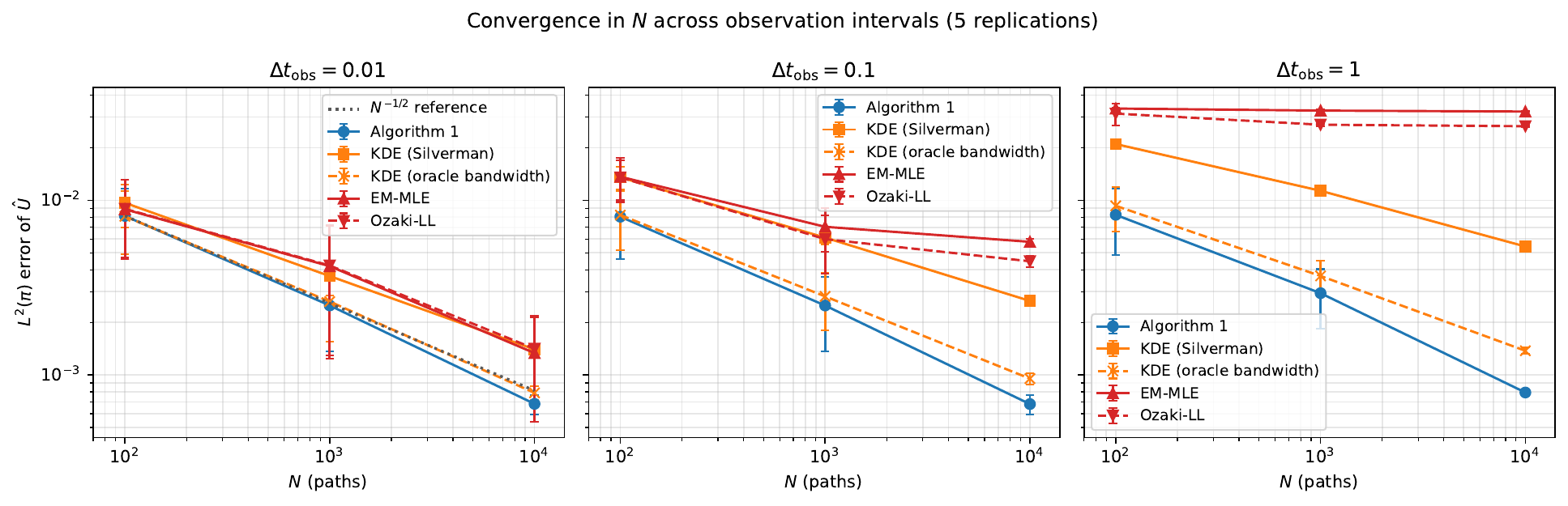}
\caption{$L^2(\pi)$ error versus $N$ at
$\Delta t_{\rm obs}\in\{0.01,0.1,1\}$. Algorithm~\ref{alg:7-1}
(blue) follows the dotted $N^{-1/2}$ reference, whereas EM-MLE (solid
red) and Ozaki-LL (dashed red) reach interval-dependent bias floors.
KDE uses Silverman's (solid orange) or oracle (dashed orange) bandwidth.
Markers and bars are means and standard deviations over five replications.}
\label{fig:phase0b-grid}
\end{figure}

Table~\ref{tab:phase0b} and Figure~\ref{fig:phase0b-grid} show that the
proposed error follows the predicted $N^{-1/2}$ rate and remains
$0.0007$--$0.0008$ at $N=10^4$ across the three intervals.  At
$\Delta t_{\rm obs}=1$, EM-MLE instead reaches a $0.0322$ bias floor,
which the Ozaki correction lowers but does not remove.

\paragraph{Isolating the observation interval.}
To separate sampling interval from observation count, we fix the total
budget at $n\approx5\times10^4$ and adjust $N$.  In
Figure~\ref{fig:phase0b-dt}, the proposed and log-density estimators do
not deteriorate as $\Delta t_{\rm obs}$ grows from $0.01$ to $5$.
The two transition-based fits instead end near $0.054$, about 30 times
the proposed estimator's error; their small replication variance at the
coarse end confirms a discretization floor.

\begin{figure}[htbp]
\centering
\includegraphics[width=0.68\textwidth]{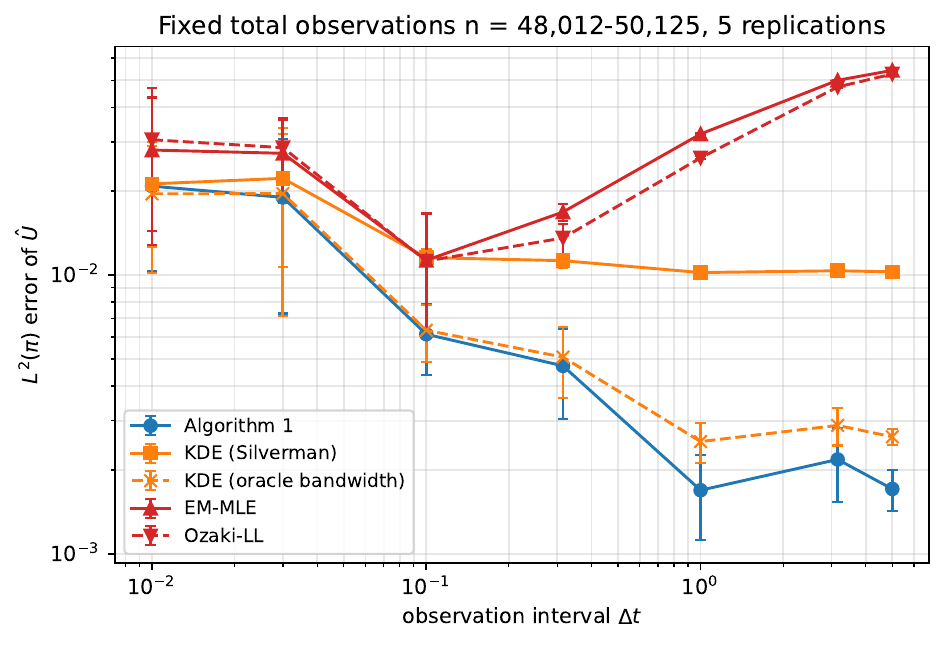}
\caption{$L^2(\pi)$ error against the observation interval at a fixed
budget $n\approx5\times10^4$, with $N$ adjusted from $12$ to $5556$.
The proposed and log-density estimators do not deteriorate, whereas the two
transition-based fits deteriorate. Markers and bars summarize five
replications.}
\label{fig:phase0b-dt}
\end{figure}

\paragraph{Basis dimension and conditioning.}
Table~\ref{tab:J-sweep} displays the regularization mechanism of
\S\ref{sec:5}: after $J=13$ the approximation error is negligible,
while regularization keeps the $L^2(\pi)$ error near
$1.5\times10^{-3}$ even as $\mathrm{cond}(B)$ reaches
$1.8\times10^{19}$.

\begin{table}[htbp]
\centering
\caption{Basis sweep on the 1D double-well ($N=1000$,
$\Delta t_{\rm obs}=0.1$, $h=0.35$, $\lambda=10^{-8}$; one realization).
The energy columns give the estimated and best-approximation errors.
For $J\gtrsim21$, $B$ is numerically singular before regularization.}
\label{tab:J-sweep}
\begin{tabular}{rrrrrr}
\toprule
$J$ & $L^2(\pi)$ error & energy error & best approximation & $\mathrm{cond}(B)$ & $\mathrm{cond}(B+\lambda I)$ \\
\midrule
 5 & 0.0178 & 0.162  & 0.162      & $1.6\times10^{4}$  & $1.6\times10^{4}$ \\
 9 & 0.0043 & 0.027  & 0.027      & $1.6\times10^{5}$  & $1.6\times10^{5}$ \\
13 & 0.0016 & 0.0052 & 0.0001     & $1.9\times10^{10}$ & $2.0\times10^{8}$ \\
21 & 0.0015 & 0.0075 & $<10^{-4}$ & $5.0\times10^{13}$ & $3.4\times10^{8}$ \\
41 & 0.0015 & 0.0076 & $<10^{-4}$ & $1.2\times10^{18}$ & $6.8\times10^{8}$ \\
81 & 0.0015 & 0.0077 & $<10^{-4}$ & $1.8\times10^{19}$ & $1.4\times10^{9}$ \\
\bottomrule
\end{tabular}
\end{table}

To exercise the full finite-horizon estimator, we use the time-dependent
potential
$U(t,x)=\bigl(1+0.4\sin(\pi t/T)\bigr)\bigl(-\tfrac14(x^2-1)^2\bigr)$,
and simulate $N=4000$ trajectories on $[0,50]$ without burn-in.  With
$\lambda=10^{-2}$, we use the tensor basis
$\psi_{mn}(t,x)=P_m(2t/T-1)\,\varphi_n(x)$ with Legendre temporal modes
$m=0,1,2$ and $21$ spatial RBFs ($J=63$), assembling $A$ and $B$ by
trapezoidal time quadrature.  Figure~\ref{fig:phase0b-timedep} shows
time-averaged relative $L^2(\pi_t)$ errors $0.036$, $0.036$, and $0.043$
at $\Delta t_{\rm obs}=0.01,0.1,1$, respectively.  This near-flat error
supports Theorem~\ref{thm:5.3} beyond stationarity.

As a self-consistency check, center each recovered time slice and define
$\rho_{\rm time}=\Big(\tfrac{1}{T}\int_0^T\|\hat U_c(t,\cdot)-\bar U_c\|_{L^2(\pi)}^2\,dt
\big/\|\bar U_c\|_{L^2(\pi)}^2\Big)^{1/2}$.  The estimate is $0.020$ on
stationary data and $0.097$ on time-dependent data, against the true
values $0$ and $0.100$.

\begin{figure}[htbp]
\centering
\includegraphics[width=0.86\textwidth]{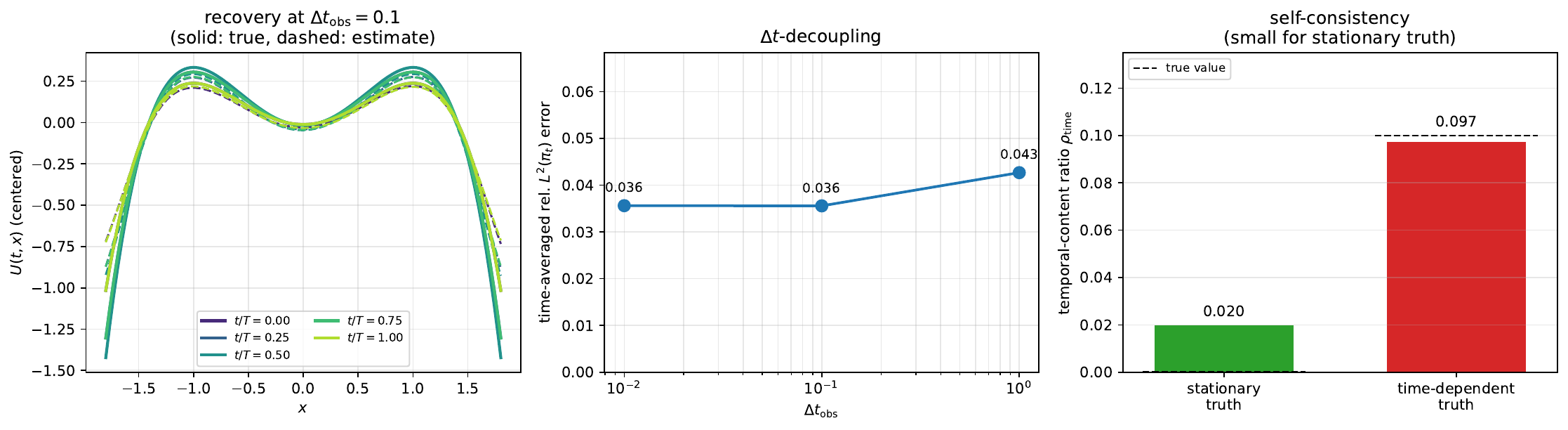}
\caption{Time-dependent 1D validation of the finite-horizon estimator.
Left: true (solid) and recovered (dashed) potentials at five times.
Middle: time-averaged relative $L^2(\pi_t)$ error versus
$\Delta t_{\rm obs}$. Right: recovered temporal-content ratio
$\rho_{\rm time}$, with true values dashed.}
\label{fig:phase0b-timedep}
\end{figure}

\subsection{Two-dimensional benchmark}\label{sec:9.2}

We extend the 1D experiment to $d=2$ on a separable double-well
\begin{equation*}
 U(x,y) = -\tfrac{1}{4}(x^2-1)^2 - \tfrac{1}{4}(y^2-1)^2,
\end{equation*}
whose stationary distribution has four modes at $(\pm 1,\pm 1)$.
The same conventions apply ($\sigma=0.4$, $T=30$,
$\Delta t_{\rm sim}=0.005$, $\Delta t_{\rm obs}=0.05$, $N=2000$ paths), with
a tensor-product Gaussian RBF basis comprising $J=13\times 13=169$ centers
on $[-2.5,2.5]^2$ at bandwidth $h=0.4$.
Figure~\ref{fig:phase1-2d} shows one representative recovery.

\begin{figure}[htbp]
\centering
\includegraphics[width=0.86\textwidth]{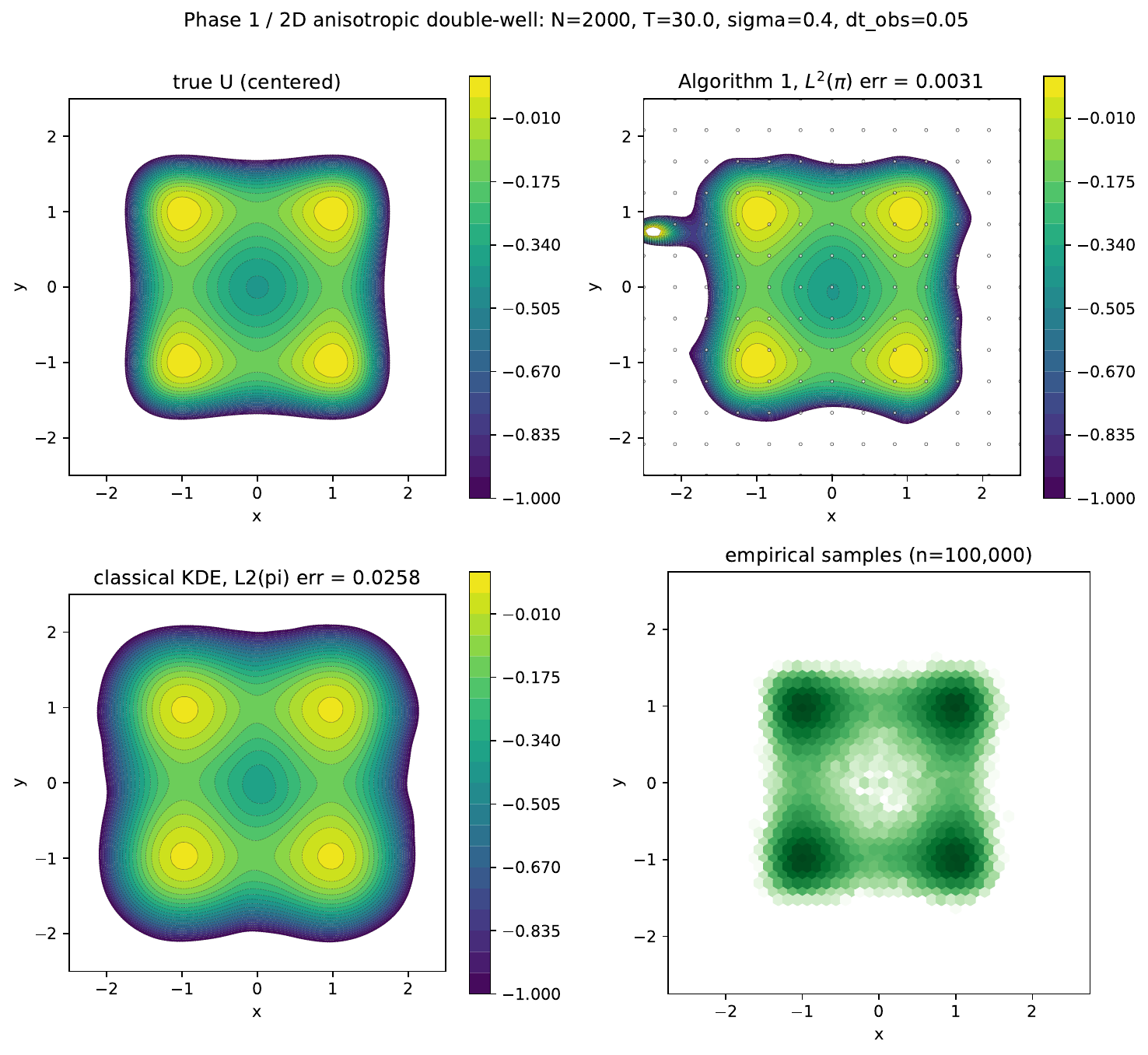}
\caption{Two-dimensional recovery on the separable double-well
($N=2000$, $T=30$, $\sigma=0.4$, $\Delta t_{\rm obs}=0.05$):
true $U(x,y)$ (top left), Algorithm~\ref{alg:7-1} estimate (top right;
white points are RBF centers), untuned log-density estimate (bottom
left), and empirical density (bottom right). The proposed estimate has
$L^2(\pi)$ error $0.0031$.}
\label{fig:phase1-2d}
\end{figure}

\paragraph{Does the gradient projection recover $U$ under rotational
drift?}
Algorithm~\ref{alg:7-1} recovers the gradient projection of the drift.
If $b=\nabla U+b_\perp$ and $\nabla\!\cdot(\pi b_\perp)=0$, then
$\int b_\perp^{\mathsf T}\nabla\varphi\,\pi\,dx
 =-\!\int\!\varphi\,\nabla\!\cdot\!(\pi b_\perp)\,dx=0$,
so the rotational component does not alter the projection.  This holds
for $b_\perp=\omega\mathcal R\nabla U$, where
$\mathcal R=\bigl(\begin{smallmatrix}0&1\\-1&0\end{smallmatrix}\bigr)$,
and preserves $\pi\propto e^{2U/\sigma^2}$.  We test the predicted
insensitivity to rotation at $\omega\in\{0,0.5,1,2,4\}$ with
$N=1000$, $T=20$, and $\Delta t_{\rm obs}=0.05$.

\begin{figure}[htbp]
\centering
\includegraphics[width=0.9\textwidth]{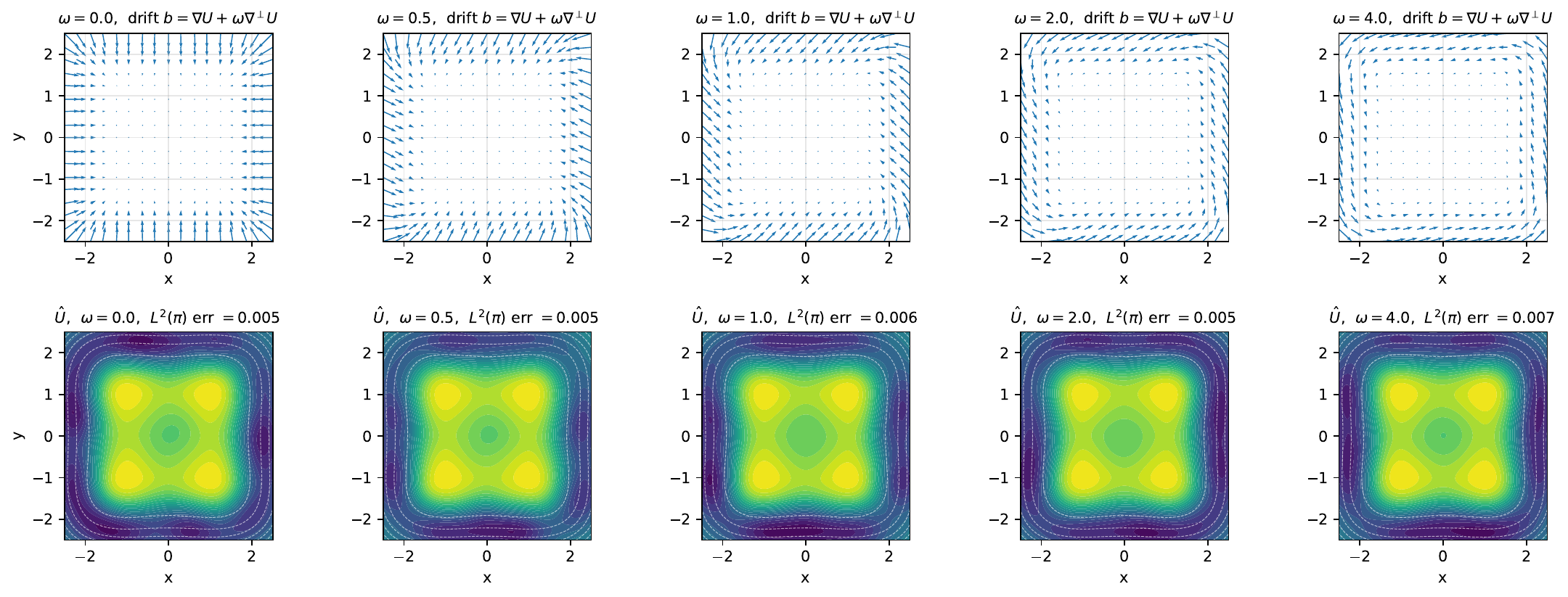}
\caption{Algorithm~\ref{alg:7-1} on the rotational-truth synthetic
$dX=(\nabla U + \omega\,\mathcal{R}\nabla U)\,dt + \sigma\,dW$.
Top: drift fields for increasing $\omega$. Bottom: recovered $\hat U$
(color) with true contours (white dashed); the surface remains stable
across the sweep.}
\label{fig:phase1-2d-rotational}
\end{figure}

\begin{figure}[htbp]
\centering
\includegraphics[width=0.6\textwidth]{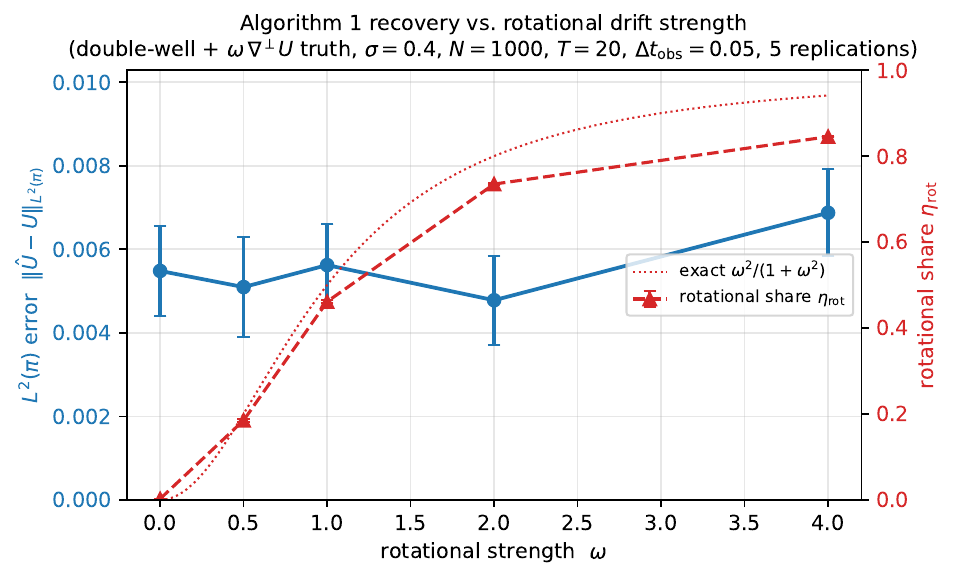}
\caption{Quantitative summary of Figure~\ref{fig:phase1-2d-rotational}.
$L^2(\pi)$ recovery error (blue) and rotational share $\eta_{\rm rot}$
(red) versus $\omega$; the dotted curve is $\omega^2/(1+\omega^2)$.
Markers and bars summarize five replications.}
\label{fig:phase1-2d-rotational-summary}
\end{figure}

Figures~\ref{fig:phase1-2d-rotational}--\ref{fig:phase1-2d-rotational-summary}
show that as $\eta_{\rm rot}$ grows from $0.00$ to $0.85$, the recovery
error stays between $0.0048$ and $0.0069$ without a monotone trend.
Thus a large rotational share need not distort the recovered gradient
component when the rotation is $\pi$-preserving.

\subsection{Illustrative application to animal telemetry with irregular sampling}\label{sec:9.3}

We illustrate the estimator on GPS telemetry from the African buffalo
\emph{Pepper} in Kruger National Park
\cite{getz-etal:2007,cross-etal:2016}, redistributed via the
\texttt{ctmm} R package \cite{ctmm}.
The record has $n=1725$ fixes over $250$ days, with median interval
$2.0$\,h and range $3$\,min--$5.1$ days.  It therefore tests the
end-to-end irregular-sampling workflow; its ecological scope is
discussed in \S\ref{sec:10.1}.

The $61\times17$\,km range is rescaled coordinatewise to $[-1,1]^2$.
We use a $13\times13$ Gaussian RBF basis with $h=0.18$ and
$\lambda=10^{-4}$; the bandwidth corresponds to approximately
$5.7$\,km by $1.6$\,km.  For a specified position-error scale $\tau$,
Proposition~\ref{prop:noise} gives an entrywise relative perturbation
of order $(\tau/h)^2$, with coefficient error also depending on the
discrete Gram-matrix stability.  The coordinatewise
rescaling avoids data-starved basis centers but induces anisotropic
noise.  With
$R=\operatorname{diag}(1/L_x,1/L_y)$,
$\Sigma=\sigma R=\operatorname{diag}(\sigma_x,\sigma_y)$, and
$\Gamma=\operatorname{diag}(\sigma_x^2,\sigma_y^2)/\bar\sigma^2$,
$\bar\sigma^2=(\sigma_x^2+\sigma_y^2)/2$, the transformed drift is
$R^2\nabla\tilde U$.  Expanding
$(\bar\sigma^2/\sigma^2)\tilde U\approx-\sum_k\alpha_k\psi_k$ gives
\begin{equation}
\label{eq:9.3-aniso-system}
   B^\Gamma\alpha
   = \tfrac{1}{2}\left[\sigma_x^2\,\hat\pi[\partial_{\tilde x\tilde x}\psi_j]
                    + \sigma_y^2\,\hat\pi[\partial_{\tilde y\tilde y}\psi_j]\right]_{j=1}^J,
   \qquad
   B^\Gamma_{jk}=\hat\pi[(\nabla\psi_j)^{\mathsf T}\Gamma\,\nabla\psi_k].
\end{equation}
The positive factor affects only the scale of the recovered surface.
Equation~\eqref{eq:9.3-aniso-system} is the case $\Sigma=\sigma R$ and
$Q=R^{-2}$ of Remark~\ref{rem:aniso}; we apply its stationary form to
the long trajectory under a quasi-stationary ergodic approximation,
assuming that the observation schedule is non-informative for location.
\paragraph{Estimating $\sigma^2$.}
We estimate the diffusion coefficient by the standard 2D
quadratic-variation formula, averaging the squared displacement over
the set $\mathcal{K}=\{k:\Delta t_k=\Delta t\}$ of increments at the
regular gap $\Delta t=3600$\,s,
\begin{equation}
\label{eq:9.3-sigma}
 \hat\sigma^2 = \frac{1}{2\Delta t\,|\mathcal{K}|}
   \sum_{k\in\mathcal{K}}\bigl(|\Delta X_k|^2+|\Delta Y_k|^2\bigr).
\end{equation}
This gives $\hat\sigma^2\approx 51\,\text{m}^2/\text{s}$ on the
$|\mathcal{K}|=400$ exactly-1-h Pepper increments.
The increment second moment contains an $O(\Delta t^2)$ correction from
the drift and local variation, so after division in \eqref{eq:9.3-sigma}
the diffusion estimate has an $O(\Delta t)$ bias whose sign is generally
model-dependent.  Independent position noise adds an upward term of
order $\tau^2/\Delta t$.  We use $\hat\sigma^2$ only to set the potential
scale: by \eqref{eq:6.3}, changing it by a common positive factor rescales
$\hat U$ without changing the recovered shape.

\paragraph{How each estimator handles irregular sampling.}
The stationary form of Algorithm~\ref{alg:7-1} (\S\ref{sec:6}) depends
on spatial empirical averages and uses all $1725$ positions.
Euler--Maruyama maximum likelihood instead uses
$\sum_k |X_{k+1}-X_k - b(X_k)\Delta t_k|^2/(2\sigma^2\Delta t_k)$.
Although the weighting accommodates irregular intervals, the
Euler transition approximation requires short steps.  We therefore
retain the $1539/1724$ increments with $\Delta t_k\le6$\,h.
Figures~\ref{fig:phase1-real-pepper-data} and
\ref{fig:phase1-real-pepper} show the sampling pattern and the recovered
surfaces.

\begin{figure}[htbp]
\centering
\includegraphics[width=0.92\textwidth]{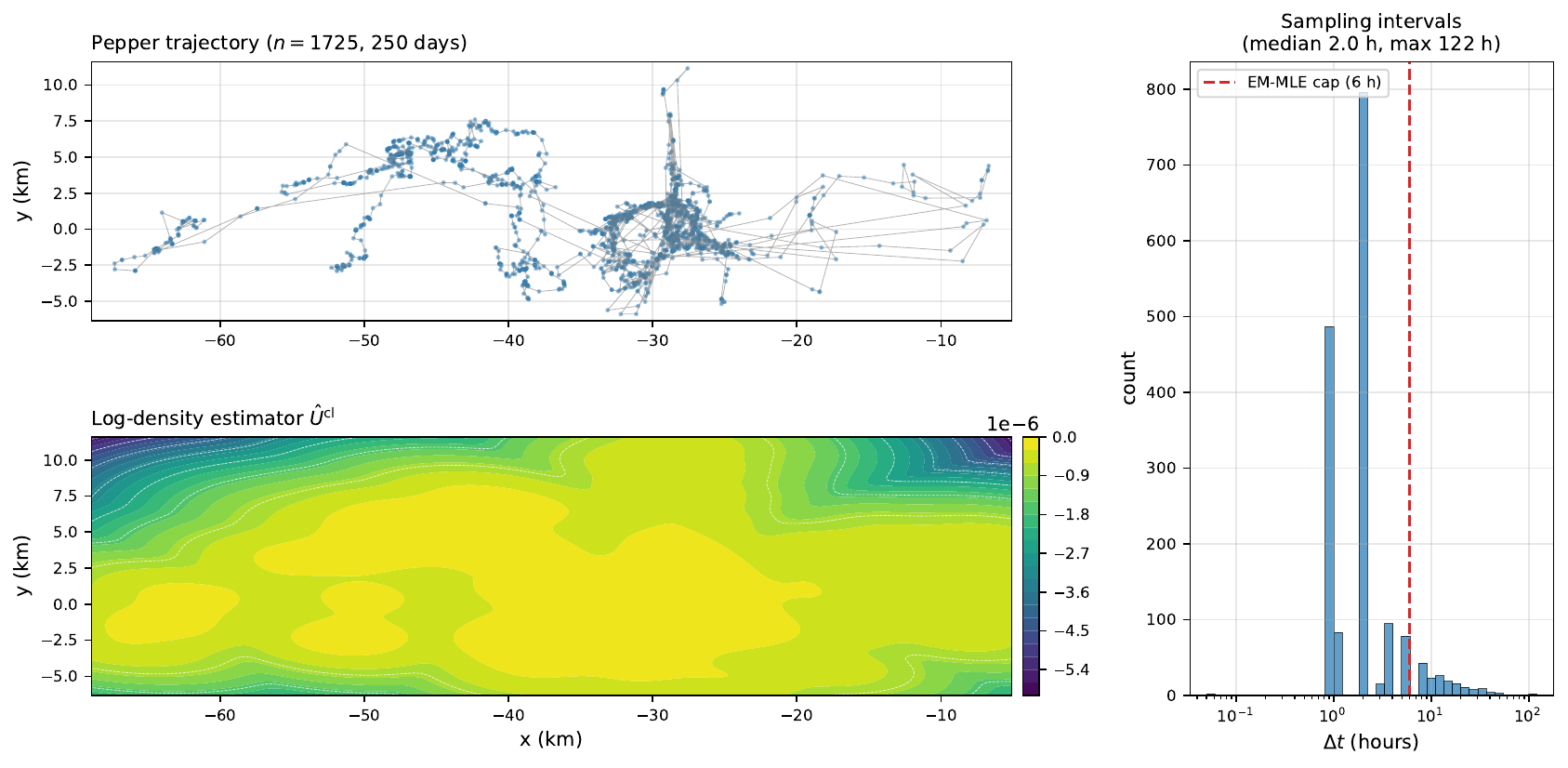}
\caption{Data context for the African buffalo \emph{Pepper} (Kruger
National Park, $n=1725$ fixes, $250$ days, irregular sampling).
Left: trajectory and log-density reference. Right: sampling-interval
histogram; the dashed line marks the $6$\,h EM-MLE cap.}
\label{fig:phase1-real-pepper-data}
\end{figure}

\begin{figure}[htbp]
\centering
\includegraphics[width=0.64\textwidth]{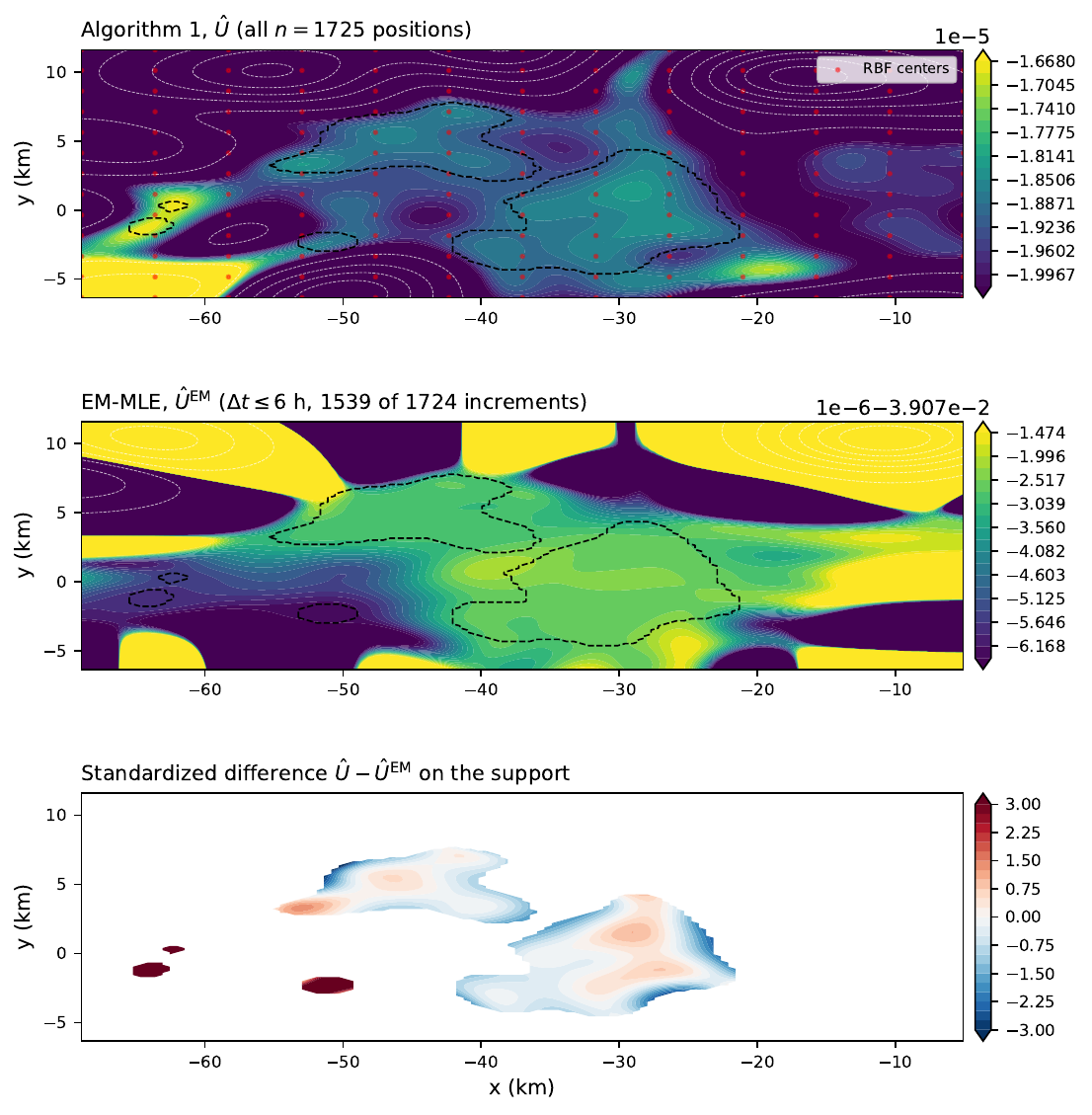}
\caption{Recovered potential surfaces on \emph{Pepper}
(same data as Figure~\ref{fig:phase1-real-pepper-data}).
Top: Algorithm~\ref{alg:7-1} using all positions. Middle: EM-MLE using
the $1539$ short-step increments. Bottom: standardized difference on
$\{\hat\pi\ge10\%\hat\pi_{\max}\}$, outlined above and clipped at
$\pm3$.}
\label{fig:phase1-real-pepper}
\end{figure}

On Pepper the EM-MLE surface is nearly constant on the occupancy
support, placing most variation outside the data region; its correlations
with $\hat U$ are therefore weak ($r=-0.27$, rank $0.21$).  The regularly
sampled buffalo \emph{Cilla} shows the same pattern ($r=0.13$, rank
$0.16$), without discarding increments.

\paragraph{Quantitative agreement with the KDE reference.}
To check that $\hat U$ is not an artifact of basis extrapolation, we
compare it to the classical log-density estimator
$\hat U^{\rm cl}=(\sigma^2/2)\log\hat\pi$.  On a $121\times121$ grid
restricted to $\{\hat\pi\ge10\%\hat\pi_{\max}\}$, we report Pearson
and Spearman correlations for all estimator pairs.  These summaries
are invariant to the unidentified additive constant and positive scale.

\begin{center}
\small
\begin{tabular}{lcccc}
\hline
 & \multicolumn{2}{c}{Pepper} & \multicolumn{2}{c}{Cilla} \\
Pair & $r$ & $\rho$ & $r$ & $\rho$ \\
\hline
$\hat U$ and $\hat U^{\rm cl}$      & $+0.37$ & $+0.60$ & $+0.67$ & $+0.73$ \\
$\hat U^{\rm EM}$ and $\hat U^{\rm cl}$ & $+0.34$ & $+0.40$ & $+0.42$ & $+0.48$ \\
$\hat U$ and $\hat U^{\rm EM}$      & $-0.27$ & $+0.21$ & $+0.13$ & $+0.16$ \\
\hline
\end{tabular}
\end{center}

On Cilla the estimate of Algorithm~\ref{alg:7-1} is closer to the
reference than the Euler--Maruyama fit is, on both measures.
On Pepper the two are close in $r$, and the rank correlation separates
them, $+0.60$ against $+0.40$.

\paragraph{Helmholtz--Hodge diagnostic of the gradient-drift assumption.}
To quantify the non-gradient component omitted by the model, we fit a
separate two-component EM-MLE drift on the same RBF basis and
short-step subset.  Its curl
$\hat\omega=\partial_x\hat b_y-\partial_y\hat b_x$ is evaluated in
physical coordinates.  Figure~\ref{fig:phase1-real-helmholtz} visualizes
this diagnostic.

\begin{figure}[H]
\centering
\includegraphics[width=0.58\textwidth]{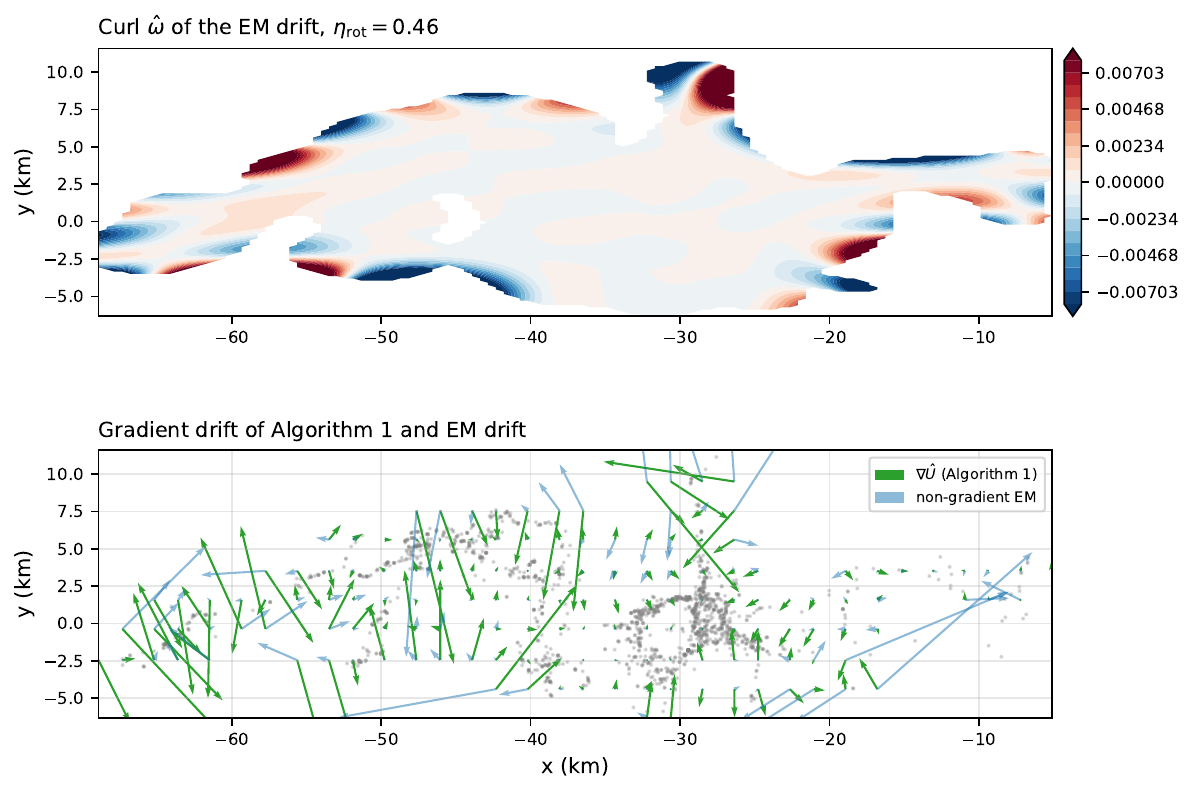}
\caption{Helmholtz--Hodge diagnostic on \emph{Pepper} (same data as
Figure~\ref{fig:phase1-real-pepper}).
Top: curl of the EM drift on the occupancy support. Bottom: gradient
drift $\nabla\hat U$ (green) and non-gradient EM drift (blue) over the
GPS positions.}
\label{fig:phase1-real-helmholtz}
\end{figure}

The scale-free rotational share is
\begin{equation}
\label{eq:9.3-curl}
 \eta_{\rm rot} = \min_{v\in\mathrm{span}\,\Psi}
   \frac{\int |\hat b(x)-\nabla v(x)|^2\,\hat\pi(x)\,dx}
        {\int |\hat b(x)|^2\,\hat\pi(x)\,dx}\in[0,1],
\end{equation}
computed on $\{\hat\pi\ge1\%\hat\pi_{\max}\}$.  Finite-basis and
support effects make it diagnostic rather than exact; the calibrated
sweep of \S\ref{sec:9.2} shows their net size.  We obtain
$\eta_{\rm rot}=0.46$ on Pepper and $0.54$ on Cilla.  Thus about half
of the fitted drift energy is rotational, but the $\pi$-preserving
experiment of \S\ref{sec:9.2} shows that this alone does not bias the
recovered gradient component.

\section{Discussion and conclusion}\label{sec:10}

\subsection{Interpretation and limitations}\label{sec:10.1}

The estimator recovers the potential $U=-v^\ast$ of
\eqref{eq:1.1}; by Proposition~\ref{prop:ls}, $\nabla\hat U$ is the
Galerkin projection of the drift in
$L^2(\tilde\mu_t(dx)\,dt)$.  Under the quadratic-control
interpretation, high values mark regions toward which the drift is
attracted \cite[Ch.~6]{hoo-etal:2017}.  The stationary algorithm is
appropriate near stationarity or on a horizon long relative to the
inverse spectral gap; otherwise the finite-horizon form of
Theorem~\ref{thm:4.3} should be used.

The control problem is a representation of the data, not a claim that
animals solve it literally.  It does not identify a unique running--
terminal cost decomposition, and marginals identify only the gradient
component \cite{weinreb-etal:2018}.  Theorem~\ref{thm:5.3} assumes
independent trajectories at common times, whereas the telemetry example
uses a quasi-stationary ergodic approximation for one path; position
noise also leaves the $O(\tau^2)$ bias of
Proposition~\ref{prop:noise}.  Thus the telemetry analysis is an
illustration, not ecological validation, and does not model individual
heterogeneity or interactions.

\subsection{Comparison with existing potential estimators}\label{sec:10.2}

The closest variational method is entropic action matching
\cite{nek-etal:2023}.  Both methods use marginal samples and, for a known
diffusion coefficient, minimize the same population functional after the sign change
$s=-v$.  Their aims and implementations differ.  Action matching starts
from a marginal curve and learns its canonical minimum-energy curl-free
drift, typically with a neural field, in order to reproduce or simulate
the observed distributional evolution.  Here the same functional is
used as an inverse stochastic-control criterion for a potential (the
negative value function).  A fixed Galerkin basis turns estimation into
the single linear solve $B\alpha=-A$ and makes the approximation,
conditioning, sampling, and time-quadrature errors explicit.  In
particular, Theorem~\ref{thm:5.3} separates the
$O(\sqrt{\log N/N})$ empirical error from the $O((\Delta t)^p)$
quadrature error on given, possibly irregular observation nodes.

The distinction is also interpretive.  If the true drift is a gradient,
the two population targets agree (up to the sign convention).  For a
general drift, marginals determine only its weighted gradient projection;
neither method identifies a divergence-free component.  We state this
projection as the estimand and interpret it as a physical potential only
under the conservative-drift assumption.  Weak-form Fokker--Planck
methods \cite{lu-etal:2024,messenger-bortz:2022} instead target drift and
diffusion coefficients directly, while \cite{batz-etal:2016} treats the
stationary linear case; Remark~\ref{rem:score-matching} records the
stationary equivalences.

For environmental covariates, \eqref{eq:4.9} has the form of a
resource-selection predictor \cite{hoo-etal:2017}.  At stationarity,
$U=(\sigma^2/2)\log\pi+C$, so the proposed, log-density, and
use-intensity estimands coincide up to constants, although their
finite-resolution errors differ.  Outside stationarity, $\hat U$ is a
finite-horizon value-function estimate, whereas the latter two summarize
occupancy.

Table~\ref{tab:methods-comparison} compares the proposed estimator with
entropic action matching, Euler--Maruyama drift fitting
\cite{bri:2010,pre-etal:2013}, KDE inversion \cite{bri-etal:2001}, and
resource- or step-selection models \cite{hoo-etal:2017}.

\begin{table}[H]
\centering
\caption{Operational comparison with selected estimators.
Symbols $\bullet$, $-$ and $\circ$ denote yes, no and setting-dependent.}
\label{tab:methods-comparison}
\renewcommand{\arraystretch}{1.0}
\small
\setlength{\tabcolsep}{4.5pt}
\begin{tabular}{lccccc}
\toprule
Property & Algo.~\ref{alg:7-1} & eAM & EM-MLE & KDE & RSF/SSF \\
\midrule
Closed form under a finite linear basis
  & $\bullet$ & $\bullet^a$ & $\bullet^b$ & n/a & $-$ \\
No transition-density approximation
  & $\bullet$ & $\bullet$ & $-$ & $\bullet$ & $\circ^c$ \\
Direct use of prescribed irregular marginal times
  & $\bullet$ & $\circ^d$ & $-$ & $\circ^e$ & $\circ^c$ \\
Linear-in-covariates potential
  & $\bullet$ & $\circ^a$ & $\bullet^b$ & $-$ & $\bullet$ \\
\midrule
Requires exogenous $\sigma^2$
  & $\bullet$ & $\bullet$ & $-$ & $\circ^f$ & $-$ \\
Restricted to a gradient/curl-free target
  & $\bullet$ & $\bullet$ & $-$ & $\bullet$ & $-$ \\
\bottomrule
\end{tabular}

\smallskip
\noindent\footnotesize
$^a$eAM uses a neural field; a finite linear parametrization gives a quadratic
system and permits linear covariates.
$^b$For an EM drift linear in its coefficients.
$^c$RSF uses positions; SSF uses steps and a step-length model.
$^d$eAM uses multiple marginal times but gives no fixed-node quadrature bound.
$^e$Stationary KDE discards interval information.
$^f$Only as the scale of $\hat U^{\rm cl}$.
\end{table}

\subsection{Summary and extensions}\label{sec:10.3}

The proposed finite-horizon Galerkin estimator turns potential recovery
into a closed-form solve whose error separates approximation, sample,
and time-quadrature contributions.  The experiments support the
predicted sampling-interval behavior and demonstrate the complete
workflow on irregular telemetry.

Natural extensions include stream-function bases for non-gradient
drifts, the kinetic Schr\"odinger problem for position--velocity tracks
\cite{chiarini-etal:2022}, and inverse mean-field control for interacting
populations \cite{santambrogio:2015}.

\subsection*{Acknowledgements}

This study is supported by JSPS KAKENHI Grant Number JP24K06861.

\noindent\textbf{AI disclosure statement.}
During the preparation of this manuscript, the author used OpenAI Codex
(accessed in 2026) to refactor and document Python code, generate routine
plotting code, and edit author-written prose.  The author critically reviewed
and tested all AI-assisted output, independently verified the mathematical
arguments, experimental design, numerical results, and conclusions, and takes
full responsibility for the content of this work.

\subsection*{Data availability}

Code reproducing all figures and the data-processing pipeline will be
available at
\url{https://github.com/yumiharu-nakano/inverse-drift-potential}.
The repository documents the provenance and licensing of the public telemetry
records distributed through
\texttt{ctmm} \cite{ctmm,fle-etal:2014,cal-etal:2016}.


\begin{thebibliography}{99}
\setlength{\itemsep}{-1.5pt}

\bibitem{ab-azar-etal:2020}
N.~Ab~Azar, A.~Shahmansoorian and M.~Davoudi,
From inverse optimal control to inverse reinforcement learning: a
historical review,
\emph{Annu.\ Rev.\ Control}, 50 (2020), 119--138.

\bibitem{batz-etal:2016}
P.~Batz, A.~Ruttor and M.~Opper,
Variational estimation of the drift for stochastic differential
equations from the empirical density,
\emph{J.\ Stat.\ Mech.\ Theory Exp.}, 2016 (2016), 083404.

\bibitem{bri:2010}
D.~R.~Brillinger,
Modeling spatial trajectories,
in \emph{Handbook of Spatial Statistics}
(A.~E.~Gelfand, P.~J.~Diggle, M.~Fuentes and P.~Guttorp, eds.),
CRC Press, 2010, 463--475.

\bibitem{bri-etal:2001}
D.~R.~Brillinger, H.~K.~Preisler, A.~A.~Ager and J.~G.~Kie,
The use of potential functions in modelling animal movement,
in \emph{Data Analysis from Statistical Foundations}
(A.~K.~M.~E.~Saleh, ed.), Nova Science, Huntington, 2001, 369--386.

\bibitem{bruckner-etal:2020}
D.~B.~Br\"uckner, P.~Ronceray and C.~P.~Broedersz,
Inferring the dynamics of underdamped stochastic systems,
\emph{Phys.\ Rev.\ Lett.}, 125 (2020), 058103.

\bibitem{cal-etal:2016}
J.~M.~Calabrese, C.~H.~Fleming and E.~Gurarie,
ctmm: an R package for analyzing animal relocation data as a
continuous-time stochastic process,
\emph{Methods Ecol.\ Evol.}, 7 (2016), 1124--1132.

\bibitem{chiarini-etal:2022}
A.~Chiarini, G.~Conforti, G.~Greco and Z.~Ren,
Entropic turnpike estimates for the kinetic Schr\"odinger problem,
\emph{Electron.\ J.\ Probab.}, 27 (2022), 152.

\bibitem{comte-etal:2007}
F.~Comte, V.~Genon-Catalot and Y.~Rozenholc,
Penalized nonparametric mean square estimation of the coefficients of
diffusion processes,
\emph{Bernoulli}, 13 (2007), 514--543.

\bibitem{cross-etal:2016}
P.~C.~Cross, J.~A.~Bowers, C.~T.~Hay, J.~Wolhuter, P.~Buss,
M.~Hofmeyr, J.~T.~du~Toit and W.~M.~Getz,
Data from: nonparameteric kernel methods for constructing home ranges
and utilization distributions,
Movebank Data Repository, 2016, DOI:10.5441/001/1.j900f88t.

\bibitem{denis-etal:2021}
C.~Denis, C.~Dion-Blanc and M.~Martinez,
A ridge estimator of the drift from discrete repeated observations of
the solution of a stochastic differential equation,
\emph{Bernoulli}, 27 (2021), 2675--2713.

\bibitem{ern-guermond:2004}
A.~Ern and J.-L.~Guermond,
\emph{Theory and Practice of Finite Elements},
Applied Mathematical Sciences, vol.~159,
Springer, 2004.

\bibitem{fle-etal:2014}
C.~H.~Fleming, J.~M.~Calabrese, T.~Mueller, K.~A.~Olson, P.~Leimgruber
and W.~F.~Fagan,
From fine-scale foraging to home ranges: a semivariance approach to
identifying movement modes across spatiotemporal scales,
\emph{Am.\ Nat.}, 183 (2014), E154--E167.

\bibitem{ctmm}
C.~H.~Fleming and J.~M.~Calabrese,
\emph{ctmm: Continuous-Time Movement Modeling},
R package version 1.3.0, 2025.
\url{https://CRAN.R-project.org/package=ctmm}.

\bibitem{florens-zmirou:1989}
D.~Florens-Zmirou,
Approximate discrete-time schemes for statistics of diffusion
processes,
\emph{Statistics}, 20 (1989), 547--557.

\bibitem{fried:1973}
I.~Fried,
Bounds on the spectral and maximum norms of the finite element
stiffness, flexibility and mass matrices,
\emph{Int.\ J.\ Solids Struct.}, 9 (1973), 1013--1034.

\bibitem{frishman-ronceray:2020}
A.~Frishman and P.~Ronceray,
Learning force fields from stochastic trajectories,
\emph{Phys.\ Rev.\ X}, 10 (2020), 021009.

\bibitem{getz-etal:2007}
W.~M.~Getz, S.~Fortmann-Roe, P.~C.~Cross, A.~J.~Lyons, S.~J.~Ryan and
C.~C.~Wilmers,
LoCoH: nonparameteric kernel methods for constructing home ranges and
utilization distributions,
\emph{PLoS ONE}, 2 (2007), e207.

\bibitem{gloaguen-etal:2018}
P.~Gloaguen, M.-P.~Etienne and S.~Le~Corff,
Stochastic differential equation based on a multimodal potential to model
movement data in ecology,
\emph{J.\ R.\ Stat.\ Soc.\ Ser.\ C}, 67 (2018), 599--619.

\bibitem{gobet-etal:2004}
E.~Gobet, M.~Hoffmann and M.~Rei\ss,
Nonparametric estimation of scalar diffusions based on low frequency data,
\emph{Ann.\ Statist.}, 32 (2004), 2223--2253.

\bibitem{hal:2016}
L.~G.~Halsey,
Terrestrial movement energetics: current knowledge and its application to
the optimising animal,
\emph{J.\ Exp.\ Biol.}, 219 (2016), 1424--1431.

\bibitem{hansen-scheinkman:1995}
L.~P.~Hansen and J.~A.~Scheinkman,
Back to the future: generating moment implications for continuous-time
Markov processes,
\emph{Econometrica}, 63 (1995), 767--804.

\bibitem{hoo-etal:2017}
M.~B.~Hooten, D.~S.~Johnson, B.~T.~McClintock, and J.~M.~Morales,
\emph{Animal Movement: Statistical Models for Telemetry Data},
CRC Press, 2017.

\bibitem{hyvarinen:2005}
A.~Hyv\"arinen,
Estimation of non-normalized statistical models by score matching,
\emph{J.\ Mach.\ Learn.\ Res.}, 6 (2005), 695--709.

\bibitem{hyvarinen:2007}
A.~Hyv\"arinen,
Some extensions of score matching,
\emph{Comput.\ Statist.\ Data Anal.}, 51 (2007), 2499--2512.

\bibitem{kays-etal:2015}
R.~Kays, M.~C.~Crofoot, W.~Jetz and M.~Wikelski,
Terrestrial animal tracking as an eye on life and planet,
\emph{Science}, 348 (2015), aaa2478.

\bibitem{kessler:1997}
M.~Kessler,
Estimation of an ergodic diffusion from discrete observations,
\emph{Scand.\ J.\ Statist.}, 24 (1997), 211--229.

\bibitem{leo:2014}
C.~L\'eonard,
A survey of the Schr\"odinger problem and some of its connections with
optimal transport,
\emph{Discrete Contin.\ Dyn.\ Syst.}, 34 (2014), 1533--1574.

\bibitem{lu-etal:2024}
L.~Lu, Z.~Zeng, Y.~Jiang, Y.~Zhu and P.~Hu,
Weak collocation regression method, fast reveal hidden stochastic
dynamics from high-dimensional aggregate data,
\emph{J.\ Comput.\ Phys.}, 502 (2024), 112799.

\bibitem{messenger-bortz:2022}
D.~A.~Messenger and D.~M.~Bortz,
Learning mean-field equations from particle data using WSINDy,
\emph{Phys.\ D}, 439 (2022), 133406.

\bibitem{mik:2021a}
T.~Mikami,
\emph{Stochastic Optimal Transportation: Stochastic Control with Fixed Marginals},
SpringerBriefs in Mathematics, Springer, 2021.

\bibitem{nak:2023}
Y.~Nakano,
Inverse stochastic optimal controls,
\emph{Automatica}, 149 (2023), 110831.

\bibitem{nak:2026}
Y.~Nakano,
Inverse stochastic control via generalized Schr\"odinger problems,
\emph{Syst.\ Control Lett.}, accepted for publication, 2026.
\url{https://arxiv.org/abs/2601.09210}.

\bibitem{nek-etal:2023}
K.~Neklyudov, R.~Brekelmans, D.~Severo and A.~Makhzani,
Action matching: learning stochastic dynamics from samples,
in \emph{Proc.\ ICML 2023}, PMLR 202, 2023, 25858--25889.

\bibitem{ng-russell:2000}
A.~Y.~Ng and S.~Russell,
Algorithms for inverse reinforcement learning,
in \emph{Proc.\ ICML 2000}, Morgan Kaufmann, 2000, 663--670.

\bibitem{ozaki:1992}
T.~Ozaki,
A bridge between nonlinear time series models and nonlinear stochastic
dynamical systems: a local linearization approach,
\emph{Statist.\ Sinica}, 2 (1992), 113--135.

\bibitem{papaspiliopoulos-etal:2012}
O.~Papaspiliopoulos, Y.~Pokern, G.~O.~Roberts and A.~M.~Stuart,
Nonparametric estimation of diffusions: a differential equations approach,
\emph{Biometrika}, 99 (2012), 511--531.

\bibitem{pavliotis-stuart:2007}
G.~A.~Pavliotis and A.~M.~Stuart,
Parameter estimation for multiscale diffusions,
\emph{J.\ Stat.\ Phys.}, 127 (2007), 741--781.

\bibitem{pokern-etal:2009}
Y.~Pokern, A.~M.~Stuart and E.~Vanden-Eijnden,
Remarks on drift estimation for diffusion processes,
\emph{Multiscale Model.\ Simul.}, 8 (2009), 69--95.

\bibitem{pre-etal:2013}
H.~K.~Preisler, A.~A.~Ager, and M.~J.~Wisdom,
Analyzing animal movement patterns using potential functions,
\emph{Ecosphere}, 4 (2013), 32.

\bibitem{ruttor-etal:2013}
A.~Ruttor, P.~Batz and M.~Opper,
Approximate Gaussian process inference for the drift function in
stochastic differential equations,
\emph{Advances in Neural Information Processing Systems}, 26 (2013).

\bibitem{santambrogio:2015}
F.~Santambrogio,
\emph{Optimal Transport for Applied Mathematicians},
Progress in Nonlinear Differential Equations and Their Applications, Vol.~87,
Birkh\"auser, 2015.

\bibitem{tenbrinck-etal:2025}
D.~Tenbrinck, N.~Uesseler, P.~Wacker and B.~Wirth,
On MAP estimates and source conditions for drift identification in SDEs,
preprint, 2025.
\url{https://arxiv.org/abs/2507.18443}.

\bibitem{vershynin:2018}
R.~Vershynin,
\emph{High-Dimensional Probability: An Introduction with Applications
in Data Science},
Cambridge University Press, 2018.

\bibitem{wainwright:2019}
M.~J.~Wainwright,
\emph{High-Dimensional Statistics: A Non-Asymptotic Viewpoint},
Cambridge University Press, 2019.

\bibitem{weinreb-etal:2018}
C.~Weinreb, S.~Wolock, B.~K.~Tusi, M.~Socolovsky and A.~M.~Klein,
Fundamental limits on dynamic inference from single-cell snapshots,
\emph{Proc.\ Natl.\ Acad.\ Sci.\ USA}, 115 (2018), E2467--E2476.

\bibitem{yon-zho:1999}
J.~Yong and X.~Y.~Zhou,
\emph{Stochastic Controls: Hamiltonian Systems and HJB Equations},
Springer, 1999.

\bibitem{ziebart-etal:2008}
B.~D.~Ziebart, A.~Maas, J.~A.~Bagnell and A.~K.~Dey,
Maximum entropy inverse reinforcement learning,
in \emph{Proc.\ AAAI 2008}, 2008, 1433--1438.

\end{thebibliography}
\end{document}